\documentclass{amsart}

\usepackage{amsmath}

\usepackage[bookmarks,bookmarksnumbered,%
  citebordercolor={0.8 0.8 0.8},filebordercolor={0.8 0.8 0.8},%
  linkbordercolor={0.8 0.8 0.8},%
  urlbordercolor={0.8 0.8 0.8},%
  ]{hyperref}
\usepackage[margin=1.0in]{geometry}
\usepackage{graphicx}       
\usepackage{amsfonts}       
\usepackage{amsthm}        
\usepackage{url}
\usepackage{amssymb}
\usepackage{latexsym}
\usepackage{bm}
\usepackage{enumerate}
\usepackage{color}
\usepackage{cancel}
\usepackage{xifthen}
\usepackage{readarray}
\usepackage{xparse}

\newtheorem{theorem}{Theorem}[section]
\newtheorem{thm}[theorem]{Theorem}
\newtheorem{prop}[theorem]{Proposition}

\newtheorem{lemma}[theorem]{Lemma}
\newtheorem{definition}[theorem]{Definition}

\newtheorem{corr}[theorem]{Corollary}

\newtheorem{conjecture}[theorem]{Conjecture}
\newtheorem{remark}{Remark}[section]

\newtheorem{question}{Question}[section]

\usepackage{thmtools}

\newcommand{\proplab}[1]{\expandafter\csname label{prp:#1}\endcsname}
\newcommand{\lmref}[1]{\expandafter\csname ref{lm:#1}\endcsname}
\newcommand{\prpref}[1]{\expandafter\csname ref{prp:#1}\endcsname}

\newcommand{\paren}[1]{\left(#1\right)}

\newcommand{\eqn}[1]{\begin{align*}&#1\end{align*}}
\newcommand{\eq}[1]{\begin{align}#1\end{align}}

\DeclareMathOperator{\R}{\mathbb{R}}
\DeclareMathOperator{\C}{\mathbb{C}}
\DeclareMathOperator{\T}{\mathbb{T}}
\DeclareMathOperator{\Q}{\mathbb{Q}}
\DeclareMathOperator{\N}{\mathbb{N}}
\DeclareMathOperator{\Z}{\mathbb{Z}}
\DeclareMathOperator{\F}{\mathbb{F}}

\DeclareMathOperator{\eps}{\epsilon}

\DeclareMathOperator{\supp}{\text{supp}}
\DeclareMathOperator{\ra}{\rightarrow}
\DeclareMathOperator{\half}{\frac{1}{2}}

\DeclareMathOperator{\sgn}{sgn}

\newcommand{\D}[2][]{\ifthenelse{\equal{#2}{}}{D(#1)}{D(#2)}}
\newcommand{\s}[2][]{\ifthenelse{\equal{#2}{}}{s(#1)}{s(#2)}}

\renewcommand{\O}[1]{O\left(#1\right)}

\newcommand{\comment}[1]{}
\newcommand{\ignore}[1]{}
\newcommand{\ignoree}[1]{\hspace{-3pt}}

\renewcommand{\S}[1][]{\ifthenelse{\isempty{#1}}{\ensuremath{\mathbb{E}}}{\ensuremath{\frac{1}{|#1|}\sum_{#1}}}}
\newcommand{\Ss}[1][]{\ifthenelse{\isempty{#1}}{\ensuremath{\mathbb{E}}}{\ensuremath{\frac{1}{|#1|}\sum_{[#1]}}}}

 {\begin{list}{}%
     {\setlength{\leftmargin}{#1}}%
     \item[]%
 }
 {\end{list}}

\newcommand{\tells}[1]{}

\newcommand{\Min}[2][]{\min_{#1}\left(#2\right)}

\newcommand{\baren}[1]{\left[#1\right]}
\DeclareMathOperator{\sign}{sign}

\let\oldexp\exp
  \renewcommand{\exp}[1]{\oldexp\left(#1\right)}

  \let\oldref\ref
  \renewcommand{\ref}[1]{(\oldref{#1})}
  
\let\oldref\ref
\renewcommand{\ref}[1]{(\oldref{#1})}

\newcommand{\norm}[1]{\left\|#1\right\|}

    \renewcommand{\D}[1]{D_w^{(#1)}}

		    \NewDocumentCommand\B{gg}{
						\IfNoValueTF{#2}
							{
							\IfNoValueTF{#1}
							{\ensuremath{\mathbb{B}}}
							{\ensuremath{\mathbb{B}_{#1} }}
							}
							{
							\IfNoValueF{#1}{\ensuremath{\mathbb{B}_{\bm{#1},\bm{#2}}}} 					      
							}					     
		     }
	  \renewcommand{\norm}[1]{\left\|#1\right\|}

	\newcommand{\abs}[1]{\left|#1\right|}
	\newcommand{\set}[1]{\left\{#1\right\}}

    \newcommand{\twincolumn}[2]{
				
				\let\Lheighttttttt\relax
				\let\LHeighttttttt\relax
				\let\Rheighttttttt\relax
				\let\RHeighttttttt\relax
				\let\TotalHeighttttttt\relax
			      \newlength{\Lheighttttttt}
			      \settoheight{\Lheighttttttt}{\vbox{
							   #1
							  }
						   }
			      \newlength{\LHeighttttttt}
			      \setlength{\LHeighttttttt}{\Lheighttttttt}
			      \newlength{\Rheighttttttt}
			      \settoheight{\Rheighttttttt}{\vbox{
							   #2
							  }
						   }
			      \newlength{\RHeighttttttt}
			      \setlength{\RHeighttttttt}{\Rheighttttttt}

			      \newlength{\TotalHeighttttttt}
			      \ifdim \LHeighttttttt < \RHeighttttttt \setlength{\TotalHeighttttttt}{\RHeighttttttt} \else \setlength{\TotalHeighttttttt}{\RHeighttttttt} \fi

				\fboxsep=0pt
			    \noindent\fbox{

			    \parbox[t][\TotalHeighttttttt]{0.48\linewidth}{ #1 }
			    
			    }%
			    \hfill%
			    \fbox{

			    \parbox[t][\TotalHeighttttttt]{0.48\linewidth}{ #2 }
			    }

			    }

    \let\oldtext\text
    \renewcommand{\text}[1]{\ \oldtext{#1}\ }

   \renewcommand{\d}[1]{2}
    
    \renewcommand{\a}[1][]{2}

\newcommand{\GL}{\text{GL}}
\newcommand{\m}{M_1}
\newcommand{\n}{M_2}
\newcommand{\adet}[1]{\abs{#1}}
\newcommand{\tran}{\intercal}


\title{On the number of 3APs in fractal sets}
\author{Marc Carnovale, Steven Senger}
\date{\today}

\begin{document}

\maketitle


\tableofcontents

\begin{abstract} We use techniques from the study of the Falconer distance conjecture to explore conditions which guarantee largeness (in terms of bounded $L^2$ density/Lebesgue measure and Hausdorff measure) of the set of lengths of step-sizes of three-term arithmetic progressions which occur within fractal sets, as well as analogous statements in discrete settings. Our main result is a version of {\L}aba and Pramanik's result in \cite{Laba} that relies only on an assumption of a lower bound, $\delta$, on the mass of the measure $\mu$ together with an upper bound, $M$ on the $L^q$ norm of its Fourier transform for some $q\in(2,3]$ depending on the parameters $\delta$ and $M$.
\end{abstract}


\section{Introduction}\label{s:intro}

A classical theorem due to Roth tells us that if we have a sufficiently dense subset of the first $N$ natural numbers, then there must be three-term arithmetic progressions (3APs) somewhere in the subset. This result has been generalized in many different ways, such as the celebrated Szemer\'edi Theorem on the presence of longer arithmetic progressions in dense subsets of natural numbers, and the Green-Tao Theorem on arithmetic progressions in the primes. Here, we give conditions ``interpolating'' between the dense and sparse-pseudorandom regimes for guaranteeing the existence of non-trivial 3APs in  subsets of $\Z$ and give criteria for finding a maximal number lengths of step-sizes of 3APs in sparse subsets of $\mathbb R^d$.

The main idea here is to apply techniques developed in the study of the Falconer Distance Problem (from \cite{falconer}) to questions of $3$AP counts in finite and fractal settings. We obtain largeness results for the set of lengths $|u|$ of step-sizes $u$ of $3$APs $\paren{x,x+u,x+2u}$ occurring within certain sets, and obtain a generalization (Theorem \ref{thm:L2FalconerResult}) of  {\L}aba and Pramanik's result in \cite{Laba} with an assumption on the $L^q$ norm of $\widehat{\mu}$ in place of a Fourier decay assumption. Our approach relies heavily on techniques and tools from the discrete world (primarily Bohr sets and the Arithmetic Regularity Lemma of Green and Tao, \cite{taogreenregularity}).

In \cite{Laba}, Laba and Pramanik proved the following theorem.

\begin{theorem}[Laba and Pramanik]\label{thm:laba-pramanik}
Assume that $E\subset[0,1]$ is a closed set which supports a probabilistic measure $\mu$
with the following properties:

\medskip

(A) $\mu([x,x+\epsilon])\leq C_1\epsilon^\alpha$ for all $0<\epsilon\leq 1$,

\medskip


(B) $|\widehat{\mu}(k)| \leq C_2 (1 - \alpha)^{-B} |k|^{-\frac{\beta}{2}}$ for all $k \ne 0$,  

\medskip\noindent
where $0<\alpha<1$ and $2/3<\beta\leq 1$.  If $\alpha>1-\epsilon_0$, where
$\epsilon_0>0$ is a sufficiently small constant depending only
on $C_1,C_2,B,\beta$, then $E$ contains a non-trivial 3-term arithmetic progression.

\end{theorem}

Their approach was powered by bounds on the operator
\[\Lambda_3(f,f,f):=\int\int f(x)f(x-r)f(x-2r)\,dx\,dr,\]
which they obtained by combining the Fourier decay estimate (B) with the Fourier representation
\[\Lambda_3(f,f,f) = \int \abs{\hat{f}(\xi)}^2\hat{f}(2\xi),\]
and controlling this via H\"older by $\hat{L}^3$ bounds, which play a privileged role in the problem. By doing so, they are able to construct a measure, $\nu$, on $[0,1]^2$ which is supported on those $(x,y)$ for which $x,y,(x+y)/2$ all lie in $\supp\mu$ (i.e., $x$ and $y$ are the endpoints of a $3$AP within $\supp\mu$), and to show that this measure gives zero mass to the trivial $3$APs where $x=y$ and itself is not the zero measure. In this paper, we take a different but equivalent approach, parametrizing by $(x,r)$ where $r$ is the common difference.

\newcommand{\si}[1]{\sigma\paren{#1}}
\newcommand{\di}[1]{\delta^{[2]}\paren{#1}}
We consider the ``difference'' measure introduced by Mattila in \cite{MattilaFalconer}, $\delta^{[1]}(\mu)$, which Mattila showed is supported on the set of distances between points $x,y$ in $\supp\mu$. We show how to extend this definition to a measure $\di{\mu}$, supported on the lengths $r=\abs{u}$ of $3$APs $(x,x+u,x+2u)\subset\supp{\mu}^3$, and how to extend the whole apparatus of Mattila's approach to this setting. One important distinction occurs: Mattila's measure $\delta^{[1]}{\mu}$ is a pushforward of $\mu\times\mu$, whereas our measure $\di{\mu}$ is not---so, as in the case of Laba and Pramanik's approach, additional care must be taken in defining the measure, and particularly, additional work is necessary to show that the measure $\di{\mu}$ has non-zero mass. It is in obtaining a lower bound on the mass of $\di{\mu}$ that we must use the Arithmetic Regularity Lemma (for $U^2([N]^d)$) as a black box, and develop technology which allows the ``interpolation'' between the bounds it provides for absolutely continuous measures with $\hat{\ell}^2$ densities and the ``pseudorandom'' bounds available to measures with small enough $\hat{\ell}^3$ norms.

This approach provides certain advantages. Firstly, Theorem \ref{thm:laba-pramanik} says nothing about how many distinct $3$APs lie in $\supp{\mu}$, while it would be desirable to know that there are many such, and that their distribution is ``fair.'' Secondly, the approach of the present article allows us to dispense with the strong requirements (A) and (B) and replace them with significantly weaker and more generic assumptions (by replacing the pointwise Hausdorff dimension (A) by an energy condition, or by replacing both (A) and (B) by an $\hat{\ell}^q$ condition). The main result of this paper is the following $L^2$-type bound under an assumption that $\widehat{\mu}\in\ell^q$ for some $q\in[2,3]$. 
 
 \begin{thm}\label{thm:L2FalconerResult}
 
 Suppose that $\mu$ is a positive measure on $[0,1]^d$ with
  \begin{enumerate}[(a)]\setcounter{enumi}{2}
   \item $\mu([0,1]^d)\geq\delta$
   \item $\norm{\widehat{\mu}}_{L^q}\leq M$
  \end{enumerate}
  for some $\delta>0$, $M<\infty$ and $q\in[2,3]$. Then there is a number\footnote{We can explicitly calculate that
  \[q_0(M,\delta) = 2 + \min_{T>1}\paren{3\paren{1-\frac{1}{T}},\frac{C_1}{\paren{C_2c^{(2)}\paren{\frac{\delta}{M}}}^{T}\ln\paren{C_2c^{(2)}\paren{\frac{\delta}{M}}}^{T}}}\]
  for some constants $C_1$ and $C_2$ and where $c^{(2)}$ is the $L^2$ Varnavides function, given in Definition \ref{def:varnfunction}.} $q_0=q_0(M,\delta)\in(2,3]$ satisfying that $\lim_{M\downarrow{ \sqrt[3]{2}\delta}} q(M,\delta) = 3$, such that if $q\leq q_0$ the set 
   \[\Delta^3(\supp\mu) = \{ r\in\R: x,x+u,x+2u\in\supp\mu, |u|=r\} \]
 has positive $s$-dimensional Hausdorff measure for all $s\leq 1$ satisfying \[s < \half+\frac{d(6-2q)}{2q},\]
 and positive Lebesgue measure when $d\geq 2$.
 
\renewcommand{\si}[1]{\sigma\paren{#1}}
\renewcommand{\di}[1]{\delta\paren{#1}}

\newcommand{\su}[2]{\int #2\,d#1}
\newcommand{\ro}{\rho}
 Further, when $d\geq 2$ the natural measure on this set derived from $\mu$ (see the definition of $\di{\mu}$ below) possesses a weighted $L^2$ density with respect to the one-dimensional Lebesgue measure.
 \end{thm}

We use the following notation to express asymptotic relationships. Given quantities $A$ and $B$ depending on some parameter or set of parameters $t$, we use the notation $A\lesssim B$ denote the existence of an unspecified constant $C$, independent of $t$, for which the inequality $A\leq C B$ holds. We use the notation $A\approx B$ to denote the simultaneous inequalities $A\lesssim B$ and $B\lesssim A$. Unless stated explicitly, we will assume that all of the 3APs considered are nontrivial, which is to say that they will consist of three distinct elements. We obtain results in two different domains which build upon and clarify one another: (1) subsets of the integers, and (2) fractal subsets of $\R^d$.

\subsection{Integer sets}
    
    In the integer setting, we will fix a large finite number $N\in \mathbb N$, and consider subsets of $[N]:= [0,\dots,N-1]$. Intuitively, it seems that if we have a large enough proportion of $[N]$, it would be difficult to avoid 3APs entirely. However, Behrend constructed a set in \cite{behrend} of size $\approx N^{1-1/\sqrt{\log N}}$ that does not contain non-trivial three term arithmetic progressions. There have been several improvements to this construction, with the current best construction of size $\approx N 2^{-\sqrt{8\log N}}$ due to O'Bryant in \cite{obryant}. In the other direction, Bloom and Sisask \cite{bloomsisask} showed that sets of density $\gtrsim (\log N)^{-1+o(1)}$ must contain 3APs; this was dramatically improved by Kelley and Meka \cite{kelleymeka}, who showed that sets avoiding 3APs have density at most $\exp{-c(\log N)^{1/12}}$ for some $c>0$. While these represent the current state of the art on results with hypotheses depending only on the size of the subset, we can say a lot more if the set possesses some quantifiable pseudorandomness. In order to make this precise, we introduce some notation.
    
    The study of arithmetic progressions has been central to additive combinatorics. The density version of the Hales-Jewett theorem, conjectured by Bergelson and Leibman \cite{bergelsonleibman} and proved by Green and Tao \cite{greentaolinearprimes}, provides a far-reaching generalization. Recent work by Fox and Pham has further elucidated the structure of sets with popular differences, a theme resonant with our investigation of the 'largeness' of the difference set.

    Suppose $E\subset [N].$ When the context is clear, we will use $E(x)$ to denote the indicator function of the set $E$, and $\hat E$ to denote the Fourier transform of the indicator function of $E.$ Let $\delta_E$ be the density of $E$, which is $|E|/N.$ One motivation for the results in this manuscript on the integers is the following version of a question of Croot.
\begin{question}[Croot and Lev, Question 7.6 of \cite{crootopenquestions}]
 Given $C,D>0$, is it true that for all sufficiently large $N$ a subset $E$ of $[N]$ satisfying
 \[\delta_E\geq\frac{1}{\log^{C}N}\]
 and
 \[\norm{\widehat{E}}_{\ell^1}\leq \log^{D}N\]
 must contain a non-trivial three-term arithmetic progression?
\end{question}

We cannot answer this difficult question. Instead, we study the easier question of how much pseudorandomness as measured by the $\ell^q$ norm of $\widehat{E}$ is necessary to guarantee $3$APs, as $q$ varies in the interval $[2,3]$. We prove the following as a consequence of a more technical result, Theorem \ref{thm:integer} in Section \ref{s:fourierlessmass}.

\begin{thm}\label{thm:integerCor}
    For any $q\in[2,3]$ and $\delta>0$ there is an $M=M(\delta,q)\in [0,\infty]$ such that for all sufficiently large $N$ a set $E\subset[N]$ of density $\geq\delta$ with $\norm{\widehat{E}}_{\ell^q}<  M$ contains  a non-trivial $3$AP, and $M$ is monotonic decreasing in its second argument with $M(\delta,3) = \sqrt[3]{2}\delta$ and $M(\delta,2)=\infty$.
  \end{thm}


\subsection{Fractal sets}

\renewcommand{\si}[1]{\sigma\paren{#1}}
\renewcommand{\di}[1]{\delta\paren{#1}}

\newcommand{\su}[2]{\int #2\,d#1}
\newcommand{\ro}{\rho}

At a rough level, an analog of Hausdorff dimension for subsets of integers $E\subset[N]$ is the smallest $\alpha$ for which $|E|\leq N^{\alpha}$. Then rephrasing the discussion at the start of Section \ref{s:intro}, Behrend's result shows that even a Hausdorff dimension of 1 is insufficient to guarantee $3$APs. However,  following Gowers' proof of Szemer\'edi's Theorem in \cite{gowers} shows that good enough Fourier decay implies that there must be 3APs. We state this as Theorem \ref{gowersThm} in Section \ref{s:fourierlessmass}.

The above interpretation of discrete results strongly suggests we ask what happens in the fractal setting, and indeed this has been studied. In 2008 Keleti constructed a full dimensional subset of $[0,1]$ containing no $3$APs (\cite{Keleti}). In 2009, \L aba and Pramanik (\cite{Laba}) proved that a probability measure $\mu$ on $[0,1]$ satisfying two conditions
\begin{equation}\label{itm:1}
\mu(B(x,r))\leq C_H r^{\alpha}
\end{equation}
\begin{equation}\label{itm:2}
\abs{\widehat{\mu}(\xi)}\leq C_F |\xi|^{-\beta/2}
\end{equation}
for $\beta>2/3$ and $\alpha$ sufficiently large depending on $\alpha$ and $\beta$ must contain $3$APs in its support. That is, they showed that some control on the Fourier transform of the measure guarantees the existence of $3$APs.

In 2016, the first listed author showed that under these conditions, $\mu$ must give positive measure to the set of starting points for $3$APs in its support (\cite{M5}). Also in 2016, \L aba, Pramanik, and Henriot (\cite{HenriotLabaPramanik}) showed, amongst other things, the 2009 \L aba-Pramanik result holds when $\beta$ is taken to be any positive number. Following these results, we conjecture the following.

\begin{conjecture}\label{thm:Conj5}
 Suppose that the compactly supported probability measure $\mu$ on $\R^d$ satisfies conditions \eqref{itm:1} and \eqref{itm:2} and that $\alpha>2d/3$. Then under appropriate quantitative assumptions on the parameters involved, the set of $r$ for which there exists an $x$ and a $u$ with $|u|=r$ such that $x,x+u,x+2u$ is a $3$AP contained within $\supp\mu$ has positive Lebesgue measure.
\end{conjecture}

While we do not prove Conjecture \ref{thm:Conj5} here, we obtain some partial results. We remind the reader of the Riesz potential,
$$I_\alpha(\mu):= \int \int |x-y|^{-\alpha}d\mu(x)d\mu(y) \approx \sum_{\xi\in\Z^d}\abs{\widehat{\mu}(\xi)}^2\abs{\xi}_{+}^{-(d-\alpha)}.$$

We need to define the following notation for the set of $r>0$ for which there exists at least one 3AP of difference with length $r$ in a given set,
$$\triangle^3(E) := \{r>0:   x,x+u,x+2u\in E, |u|=r\}.$$

 Note that studying the size of $\triangle^3(\supp\mu)$ as a notion of ``largeness'' for the parameter set 
 \[\set{r:\text{there exists $x$ such that} x,x+r,x+2r\in\supp\mu}\]
 seems natural in light of Falconer's Distance Conjecture, which conjectures that the 
 \[\triangle^2(\supp\mu) := \set{r : \text{there exists $x$ and a $u$ with $|u|=r$ such that} x,x+u\in\supp(\mu)}\] have positive Lebesgue measure whenever the Hausdorff dimension, $\alpha$, of $\mu$ is greater than $d/2$ and $d\geq 2$. A sharpness result of Falconer (\cite{falconer}) shows that dimension $d/2$ is best possible, and Mattila (\cite{MattilaFalconer}) notes that a weaker Fourier decay condition than in \eqref{itm:b} (stated below) guarantees the truth of the conclusion of the Falconer Distance Conjecture for a measure $\mu$ when $\alpha+\beta>d$ (an approach which has been central to progress since, including recent breakthroughs by Guth--Iosevich--Ou--Wang \cite{guthiosevichouwang} and work of Shmerkin \cite{shmerkin2022} and Shmerkin--Wang \cite{shmerkinwang} on self-similar sets). As necessarily the set of step-sizes of $3$APs within $\supp\mu$ is contained within the set of differences between points in $\supp\mu$, optimal size estimates on the former under a given set of hypotheses is strictly harder than for the latter. For related work on avoiding patterns in large sets, see Fraser--Pramanik \cite{fraserpramanik}.
 
 In addition to Theorem \ref{thm:L2FalconerResult}, we obtain the following simpler variant in terms of the $s$-dimensional energy, which may seem more natural to those familiar with the {\L}aba-Pramanik result \cite{Laba}. To state our result, we define relaxations of conditions \eqref{itm:1} and \eqref{itm:2} suited for our purposes:
\begin{enumerate}[(a)]
\item\label{itm:a} $I_{\alpha}(\abs{\mu})<\infty$ 
\item\label{itm:b} $|\widehat{\abs{\mu}}(\xi)|\leq C_F (1+|\xi|)^{-\frac{\beta}{2}} \text{ for all } \xi\in\Z^d$.
\end{enumerate}

 \begin{thm}\label{thm:MAIN}
   Let $\mu$ be a Radon probability measure on $\T^d$ satisfying \eqref{itm:a} and \eqref{itm:b}. Then if 
   \[d-1>(4d-3\beta)/2,\]
  \[ (2\beta + d-\alpha)/\beta < 3,\]
  and $\alpha$ is sufficiently close\footnote{
 Throughout this paper we will often use phrases of the form ``$\alpha$ is sufficiently close to $d$ depending on the parameters [$\beta,C,X,Y$, etc.]'' in situations where satisfyingly explicit bounds are unavailable. In the case at hand, for instance, what can be said is the following. In Theorem \ref{thm:MAIN}, provided that $\alpha\geq \alpha_0,  I_{\alpha}(\mu)\leq I_{\alpha_0},$ and
\[\sup_{\xi\in\Z^d} \frac{\abs{\widehat{\mu}(\xi)}^2}{\abs{\xi}^{\beta}}\leq C_F\]
 for some bounds $0<\alpha_0,I_{\alpha_0}<\infty$, and 
 \[c_0:= c\paren{ \sqrt{I_{\alpha_0}}C_F2^{N(d-\alpha_0)}}\]
 where $c$ is the Varnavides function, it suffices that 
 \[\frac{12C_F I_{\alpha_0}}{\paren{1-2^{-(\beta_0 - 2(d-\alpha_0))/2}}} 2^{-(\beta_0 - (d-\alpha_0))N/2}<c_0\]
 for some $N\in\N$. We will see in the proof that by increasing $\alpha_0$ towards $d$ such an $N$ may always be found provided that $\beta_0>2(d-\alpha_0)$.}
to $d$ depending on the quantities $C_F,\beta$, and $I_{\alpha}(\mu)$, the bilinear distance set $\triangle^3(\supp \mu)$ has positive Lebesgue measure.
 \end{thm}

 {\L}aba-Pramanik type results like Theorem \ref{thm:MAIN} follow directly from Theorem \ref{thm:L2FalconerResult}, as the following corollary makes precise. This is one indication of the naturality of the $\ell^q$-norm condition on $\widehat{\mu}$.

 \begin{corr}\label{thm:Fractal3AP}
  Suppose that the positive measure $\mu$ satisfies \eqref{itm:a} and \eqref{itm:b} where $\alpha$ is sufficiently close to $d$ depending on the parameters $I_{\alpha}(\mu)$, $C_F$, and $\beta$. Then $\supp\mu$ contains non-trivial $3$APs, and indeed the set 
   \[\Delta^3(\supp\mu) = \{ r\in\R: x,x+u,x+2u\in\supp\mu, |u|=r\} \]
  has positive $s$-dimensional Hausdorff measure for any $s\leq 1$ satisfying $s< \half + \frac{d(2\beta+4\alpha-4d)}{4(\beta+d-\alpha)}$, and has positive Lebesgue measure if $d\geq 2$.
 \end{corr}
 \begin{proof}
  By Lemma \ref{thm:decayToLq}, we have that  for $q_1=2\paren{\frac{\beta+ d -\alpha}{\beta}}$
  \[\norm{\widehat{\mu}}_{\ell^{q_1}}^{q_1}\leq I_{\alpha}(\mu)C_F^{q_1-2} \sup_{\xi}\abs{\xi}^{-\paren{q_1-2}\beta/2}\abs{\xi}^{d-\alpha}=I_{\alpha}(\mu)C_F^{q_1-2},\]
  or
  \eq{\label{good}\norm{\widehat{\mu}}_{\ell^{q_1}}\leq \paren{\sqrt{I_{\alpha}(\mu)}}^{\frac{\beta}{\beta+d-\alpha}}C_F^{\frac{d-\alpha}{\beta+d-\alpha}}.}
  
  Consider the function $q_0(M):=q(M,1)$ of Theorem \ref{thm:L2FalconerResult}. If we set $M=M(I_{\alpha}(\mu),C_F,\alpha,\beta)$ to be the right-hand-side of \eqref{good}, then we see that as $I_{\alpha}(\mu)$ stays bounded above by a parameter $C_H$, and as $d-\alpha\downarrow 0$, $M\downarrow \sqrt{I_{\alpha}(\mu)}$ and so $q_0(M)$ nears some $q_0\in(2,3)$. Meanwhile, $q_1=2\frac{\beta+d-\alpha}{\beta}\downarrow 2$, so if $d-\alpha$ is small enough, $q_1\leq q_0$ and the conclusion of Theorem \ref{thm:L2FalconerResult} holds for the measure $\mu$.
\end{proof}

\subsection*{Organization of the paper}
The paper is organized as follows. In Section~\ref{s:fourierlessmass}, we develop the integer results, proving Theorem~\ref{thm:integerCor} as a consequence of the more technical Proposition~\ref{thm:LpMassLemma}. The key tools are a Bohr-set decomposition (Lemma~\ref{thm:interpolate}) and an $L^2$ arithmetic regularity lemma adapted from Green--Tao \cite{taogreenregularity}. In Section~\ref{s:fractalresults}, we adapt Mattila's distance set machinery \cite{MattilaFalconer} to the setting of $3$APs, proving Theorems~\ref{thm:MAIN} and~\ref{thm:L2FalconerResult}. The main components are a Support Lemma, a Mass Lemma (which uses the integer results as a black box), and a Non-singularity Lemma that controls the $L^2$ density of the $3$AP length measure.

\newcommand{\Ff}{\mathbb{F}}
\renewcommand{\si}{\sigma_{E}}
\renewcommand{\di}{\nu}


 \section{Integer results}\label{s:fourierlessmass}
 \subsection{Background and tools}
 
 Let $\Z_{N} := \Z/N\Z$
endowed with normalized counting measure. Given a compact abelian group $G$ endowed with normalized Haar measure $m$ and with character group $\widehat{G}$, let $\hat{f}$ refer to the Fourier transform $\hat{f}:\widehat{G}\ra\C$ given by
 \[\hat{f}(\xi) = \int_{G} f(x)\xi(x)\,dm(x).\]
 
 Define also 
 \[\norm{f}_{U^2(G)} := \paren{\abs{\int f(x)\overline{f(x-u)}\,dm(x)}^2\,dm(u)}^{1/4} = \paren{\sum_{\xi\in \widehat{G}} \abs{\widehat{f}(\xi)}^4}^{1/4}.\]
 
\begin{theorem}[Gowers \cite{gowers}]\label{gowersThm}
If $E$ has sufficient Fourier decay that 
\[ \norm{E(x)-\delta_{E}}_{U^2}^4 = \sum_{\xi\neq 0\in\Z_N} \abs{\widehat{E}(\xi)}^4\leq \half \delta_E^3\]
then much smaller densities, $\delta_E$, suffice for the set to contain non-trivial 3APs---such a set $E$ contains non-trivial $3$APs once $\delta_E$ is greater than $N^{-\frac{1}{2}}.$
\end{theorem}

Toward this end, we record some more background material requisite for this and other results in the paper. We begin with the celebrated Varnavides' Theorem, from \cite{varn}.
 \begin{theorem}\label{thm:varnThm}[Varnavides' Theorem.]
 Let $0<\delta<1.$ Then there exists a constant, $N=N(\delta)>0$ such that for any subset $E\subset \{0,1,\dots, N\}^d,$ of size $|E|>\delta N^d$, there are $\gtrsim N^{2d}$ three-term arithmetic progressions in $E$.
 \end{theorem}
 
 To get a quantitative handle on this result, we introduce some related objects.
  
  \begin{definition}
   Let $G$ be a compact abelian group with Haar measure $m$. We define the multilinear functional $\Lambda_3: \paren{L^{\infty}(G)}^3\ra\C$ via the formula
   \[\Lambda_3(f_0,f_1,f_2)= \iint f_0(x)f_1(x-r)f_2(x-2r)\,dm(x)\,dm(r).\]
   Given a single $f\in L^{\infty}(G)$, we define
   \[\Lambda_3(f) = \Lambda_3(f,f,f).\]
  \end{definition}
  
\begin{definition}\label{def:varnfunction}
 The Varnavides function $c:(0,\infty)\ra(0,\infty)$ is given by
 \[c(t) = \inf\set{\Lambda_3(f) : f\geq 0, \norm{f}_{L^1(\T)}=1, \norm{f}_{L^{\infty}(\T)}\leq t}.\]
 
 The $L^2$ Varnavides function is the function $c^{(2)}:(0,\infty)\ra(0,\infty)$ given by
 \[c^{(2)}(t) = \inf\set{\Lambda_3(f) : f\geq 0, \norm{f}_{L^1(\T)}\geq t, \norm{f}_{L^2(\T)}=1}.\]
\end{definition}

A discretization argument together with Theorem \ref{thm:varnThm} gives that the Varnavides function $c$ above does indeed map $(0,\infty)$ to $(0,\infty)$, while the same statement for the function $c^{(2)}$ will be given as Corollary \ref{thm:L2count} in Section \ref{s:fourierlessmass}. We summarize the equivalent functional form of the first inequality below.
 
 \begin{corr}\label{thm:varnavides}
 Let $X\in\set{\T^d,\Z_N^d}$. Then for any $M,\delta>0$ there exists a constant $c=c(M,\delta)>0$ with the following property. Let $f:X\ra[0,M]$ satisfy $\norm{f}_{L^1(X)}\geq\delta$. Then
 \[\Lambda_3(f)\geq c.\]
 \end{corr}
 
 Throughout, given an element $r\in\T$, we will write
 \[\abs{r}= \abs{r}_{\T} = \min\{r \text{(mod} 1), (1-r)\text{(mod} 1)\},\]
 and similarly for $r=(r_1,\dots,r_d)\in\T^d$ we define
 \[\abs{r}=\abs{r}_{\T^d}=\paren{\sum_{i=1}^d \abs{r_i}_{\T}^2}^{\half}.\]

\subsection{Proofs}
 
A non-zero positive measure $\mu$ on $\T^d$ with $\widehat{\mu}\in\ell^2$ must necessarily contain within its support non-trivial three-term arithmetic progressions: the qualitative statement follows from the observation that since $\mu\in L^2$, $\supp\mu$ is a set of positive Lebesgue measure, together with the statement that such sets contain non-trivial $3$APs (and indeed affine images of any finite point configuration), a consequence of the Lebesgue density theorem and the pigeonhole principle. The quantitative statement is the following consequence of the arithmetic regularity lemma, whose discussion we postpone until Subsection \ref{s:reglemma}.
  
 \begin{corr}\label{thm:L2count}
   Suppose that $X=\T^d$ or $\Z_N^d$, and that $f:X\ra[0,\infty)$ satisfies $\norm{f}_{L^1}\geq\delta,\norm{f}_{L^2} < M< \infty$. Then there exists a number $c=c^{(2)}(M,\delta)$ such that 
  \[\Lambda_3(f)\geq c.\]
 \end{corr}
 
 On the other hand, if $\widehat{\mu}\in\ell^3$ with sufficiently small norm, then again lower bounds on $\Lambda_3(\mu)$ are available.

  \begin{lemma}\label{thm:L3count}
   Suppose that $\mu$ is a positive measure on $\T^d$ or $\Z_N^d$ with mass $\norm{\mu}\geq \delta$ and $\norm{\widehat{\mu}}_{\ell^3}\leq \sqrt[3]{1+c}\hspace{1pt}\delta$. Then
  \[\Lambda_3(\mu)\geq (1-c)\delta^3.\]
  \end{lemma}
  \begin{proof}
   We have
   \[\paren{\sqrt[3]{1+c}\hspace{1pt}\delta}^3 \geq \norm{\widehat{\mu}}_{\ell^3}^3 = \widehat{\mu}(0)^3+\sum_{\xi\neq 0} \abs{\widehat{\mu}(\xi)}^3 \geq \delta^3+\sum_{\xi\neq 0} \abs{\widehat{\mu}(\xi)}^3\]
   so
   \[\norm{\widehat{\mu}}_{\ell^3(\xi\neq 0)}\leq \sqrt[3]{c}\delta.\]
   Further, note by H\"older that
   \[\sum_{\xi\neq 0} \abs{\widehat{\mu}(\xi)^2\widehat{\mu}(-2\xi)}\leq \paren{\sum_{\xi\neq 0} \abs{\widehat{\mu}(\xi)}^3}^{2/3}\paren{\sum_{\xi\neq 0} \abs{\widehat{\mu}(2\xi)}^3}^{1/3}\leq \norm{\widehat{\mu}}_{\ell^3(\xi\neq 0)}^3.\]
   Then
   \[\abs{\Lambda_3(\mu) -\delta^3}= \abs{\sum_{\xi} \widehat{\mu}(\xi)^2\widehat{\mu}(-2\xi)-\delta^3} = \abs{\sum_{\xi\neq 0} \widehat{\mu}(\xi)^2\widehat{\mu}(-2\xi)} \leq c\delta^3.\]
  \end{proof}

  In summary: If $\mu$ is a positive, non-zero measure on $\T^d$ and $\norm{\widehat{\mu}}_{\ell^p}= M$, then if $M$ is sufficiently small and $p=3$ then $\Lambda_3(\mu)>0$, while if $p=2$ then no bound whatsoever on $M$ is required to conclude that $\Lambda_3(\mu)>0$. It is natural to ask whether one can interpolate between these results for $p\in(2,3)$. Indeed we can, as Proposition \ref{thm:LpMassLemma} below shows.

  \begin{prop}[Mass lemma from $\hat{\ell}^p$ bounds]\label{thm:LpMassLemma}
  \label{thm:MassLemmaWithoutFourierDecay}
   Let $\delta,M>0$. Then there exists a $q=q(M,\delta)$ and a number $c=c^{(q)}(M,\delta)>0$ with 
   \[q(M,\delta) = 2 + \min_{T>1}\paren{3\paren{1-\frac{1}{T}},\frac{C_1}{\paren{C_2c^{(2)}\paren{\frac{\delta}{M}}}^{T}\ln\paren{C_2c^{(2)}\paren{\frac{\delta}{M}}}^{T}}}\]
   for some constants $C_1,C_2$ and where $c^{(2)}(\frac{\delta}{M})$ is the $L^2$ Varnavides function,
   such that the following holds.
    Suppose that $\mu$ is a positive measure on $[0,1]^d$ or $[N]^d$ with $\norm{\mu}\geq\delta$ and $\norm{\widehat{\mu}}_{\ell^q}\leq M$. Then 
    \[\Lambda_3(\mu)\geq c^{(q)}(M,\delta).\]
    
    Equivalently, for any $q\in[2,3]$ there is an $M=M(q)$ and a $c=c^{(q)}(M,\delta)>0$ such that if $\norm{\mu}\geq\delta$ and $\norm{\widehat{\mu}}_{\ell^q}\leq M$ then $\Lambda_3(\mu)\geq c$, with $M(q)\uparrow \infty$ as $q\downarrow 2$.
  \end{prop}
  
  Theorem \ref{thm:integer} follows as a corollary.  Again, see Definition \ref{def:varnfunction} for the definition of the $L^2$ Varnavides function $c^{(2)}$.
  
    \begin{thm}\label{thm:integer}
   Let $E\subset[N]$ be a set with density $\delta=\delta_E\gg N^{-\half}$ which satisfies a bound
   \[\norm{\widehat{E}}_{\ell^q}\leq \delta M\]
   where 
   \[q\leq q(M,\delta) := 2 + \min_{T>1}\paren{3\paren{1-\frac{1}{T}},\frac{C_1}{\paren{C_2c^{(2)}\paren{\frac{\delta}{M}}}^{T}\ln\paren{C_2c^{(2)}\paren{\frac{\delta}{M}}}^{T}}}.\]
  
   Then provided $N$ is sufficiently large depending on $q$ and $M$ the set $E$ contains non-trivial $3$APs. Alternatively, for any $q\in[2,3]$ and $\delta>0$ there is an $M=M(\delta,q)\in [0,\infty]$ such that for all sufficiently large $N$ a set $E\subset[N]$ of density $\geq\delta$ with $\norm{\widehat{E}}_{\ell^q}<  M$ contains  a non-trivial $3$AP, and $M$ is monotonic decreasing in its second argument with $M(\delta,3) = \sqrt[3]{2}\delta$ and $M(\delta,2)=\infty$.
  \end{thm}
  \begin{proof}[Proof of Theorem \ref{thm:integer}.]
   Set $\mu=\frac{1_E}{\delta}$. Then $\norm{\widehat{\mu}}_{\ell^q}\leq M$ and $\norm{\mu}=1$, so by Proposition \ref{thm:MassLemmaWithoutFourierDecay}, 
   \[\Lambda_3(\mu)\geq c^{(q)}(M,1)\]
   independent of $N$.
   
   On the other hand, the trivial $3$APs within $E$ contribute
   \[\frac{1}{N^2}\sum_{x\in[N]}\prod_{i=0}^2 \mu(x) \mu(x) \mu(x) = \frac{1}{N^2} \frac{|E|}{\delta^3}=\frac{N}{|E|^2}=\delta^{-2}N^{-1}=o_{N}(1)\]
   since $\delta\gg N^{-\half}$. For large enough $N$, this will be less than $c^{(q)}(M,1)$, guaranteeing the existence of a non-trivial $3$AP within $E$.
  \end{proof}

  \begin{remark}
    In finite field settings, analogous questions arise. The arithmetic regularity lemma and inverse theorems (see \cite{greentaoinverse}) suggest that Fourier decay assumptions and assumptions on counts of $3$APs are morally equivalent in appropriate quantitative regimes. We explore this direction in forthcoming work.
  \end{remark}

 The proof of Proposition \ref{thm:LpMassLemma} will use the following lemma, whose proof we delay until the end of the present section.
  \begin{lemma}\label{thm:interpolate}
  For all $\eps\in(0,1)$ and $C>1$ there exists a $q=q(\eps,C)\in(2,3)$ such that for all finite borel measures $f$ with $\hat{f}\in\ell^q(\Z^d)$ there exists a decomposition
  \[f = g+h\]
  such that
  \begin{itemize}
   \item $g\geq 0$,
   \item $\norm{g}_{L^1}=\norm{f}$,
   \item $\big\lvert\hat{g}\big\rvert,\big\lvert{\hat{h}}\big\rvert\lesssim\big\lvert{\hat{f}}\big\rvert$,
   \item $\big\|{\hat{g}}\big\|_{\ell^2}\lesssim C\big\|{\hat{f}}\big\|_{\ell^q}$,
   
   and
   \item $\big\|{\hat{h}}\big\|_{\ell^3}\lesssim\eps$.
  \end{itemize} 
  Further, if we specify an upper bound $M$ for the $\ell^q$ norm of $\hat{f}$, we may take
  \eq{\label{pvalue}q = 2 + \min_{T>1}\paren{3\paren{1-\frac{1}{T}},\frac{2\ln C}{\paren{\frac{M}{\eps}}^{T}\ln\paren{\frac{M}{\eps}}^{T}}}.}
  \end{lemma}
  

  \begin{proof}[Proof of Proposition \ref{thm:LpMassLemma}]
  Set the measure to be $f=\mu$. Fix a number $\eps>0$ depending on $M$ and $\delta$ to be specified later. We use Lemma \ref{thm:interpolate} with $C\approx 4$ to find a $q=q(M,\eps)\in(2,3)$ and a decomposition
  \[f = g+h\]
  where
  \eq{
   & q = 2 + \min_{T>1}\paren{3\paren{1-\frac{T-1}{T}},\frac{C'}{\paren{\frac{M}{\eps}}^{T}\ln\paren{\frac{M}{\eps}}^{T}}}
   \\&\notag{}g\geq 0
   \\&\notag{} \abs{\hat{g}},\abs{\hat{h}}\lesssim\abs{\hat{f}},
   \\&\notag{} \norm{\hat{g}}_{\ell^2}\leq 4\norm{\hat{f}}_{\ell^q}\leq 4 M,
   \\&\notag{}\norm{\hat{g}}_{\ell^3}\leq\norm{\hat{f}}_{\ell^3}\leq\norm{\hat{f}}_{\ell^q}\leq 4M,
   \\&\label{goodbounds} \norm{\hat{h}}_{\ell^3}\lesssim\eps.
  }
 
 Rescaling $g$ by $(4M)^{-1}$, Corollary \ref{thm:L2count}  gives
 \eq{\label{gcount}\Lambda_3(g)\geq \paren{4M}^3c^{(2)}\paren{\frac{\delta}{4M}}.}
 
 We have
 \[\abs{\Lambda_3(f)-\Lambda_3(g)}\lesssim \abs{\Lambda_3(f_0,f_1,f-g)}\leq \sum \abs{\hat{f_0}(\xi)}\abs{\hat{h}(\xi)}\abs{\hat{f_1}(-2\xi)}\leq \norm{\hat{h}}_{\ell^3}\norm{\hat{f_0}}_{\ell^3}\norm{\hat{f_1}}_{\ell^3}\]
  where $f_0,f_1\in\set{f,g}$ are chosen to maximize the estimate on the right-hand side of  the first inequality. Thus applying \eqref{goodbounds}
  \[\abs{\Lambda_3(f)-\Lambda_3(g)}\lesssim \eps M^2\]
  whence using \eqref{gcount} we have
    \[\Lambda_3(f)\geq 4M^3c^{(2)}\paren{\delta/4M}-\O{\eps M^2}.\]
   
   So choosing 
   \[\eps = C''4Mc^{(2)}\paren{\delta/4M}\]
   for sufficiently small constant $C'$ gives the result with
   \[c^{(q)}(M,\delta) \gtrsim M^3c^{(2)}\paren{\delta/4M}.\]
  \end{proof}


\begin{proof}[Proof of Lemma \ref{thm:interpolate}]
  %
  
  Fix a cut–off $\lambda\in(0,\|f\|_{L^{1}})$ and a parameter
  \(\eta\in(0,\tfrac12)\) to be chosen later.
  Let \(q\in(2,3)\) (also to be chosen later) and suppose
  \(\|\hat f\|_{\ell^{q}}\le M\).
  
  \medskip
  \noindent
  {\bf Step 1: Bohr-cut decomposition.}
  Put
  \[
    E_{\lambda}\;:=\;\bigl\{\,\xi\in\widehat{\Bbb G}\ :\ |f(\xi)|\ge\lambda\bigr\},
  \]
  and let \(B=B(E_{\lambda},\eta)\) be the Bohr set
  \[
    B(E_{\lambda},\eta)
    \;=\;
    \bigl\{x\in\Bbb G\ :\
          |e^{2\pi i\xi\cdot x}-1|\le\eta
          \text{ for every }\xi\in E_{\lambda}\bigr\}.
  \]
  Write \(\phi:=\frac1{|B|}\,1_{B}\) and set
  \[
    g := \phi*f,
    \qquad
    h := f-g.
  \]
  
  \medskip
  \noindent
  {\bf Step 2.1: Fourier‐norm bounds.}
  As \(|E_{\lambda}|\le\lambda^{-q}\|\hat f\|_{\ell^{q}}^{q}\),
  Lemma 4.4 of \cite{taovu} gives, for every \(p\ge2\),
  \begin{equation}\label{fourierlqbohrbound}
    \|\hat\phi\|_{\ell^{p}}
    \;\lesssim\;
    \eta^{-\lambda^{-q}\|\hat f\|_{\ell^{q}}^{q}/p}
    \;\le\;
    \eta^{-(M/\lambda)^{q}/p}.
  \end{equation}
  
  \noindent
  {\bf Step 2.2: Bound for \(g\).} By H\"older and
  \eqref{fourierlqbohrbound} with \(p=\frac{2}{q-2}\),
  \begin{equation}\label{gbound}
    \|\,\hat g\|_{\ell^{2}}
    \;=\;
    \Bigl(\sum|\hat\phi|^{2}|\hat f|^{2}\Bigr)^{1/2}
    \;\lesssim\;
    M\,\eta^{-(M/\lambda)^{q}(q-2)/2}.
  \end{equation}
  
  \noindent
  {\bf Step 2.3: Bound for \(h\).} Using
  \(|1-\hat\phi(\xi)|\lesssim\eta^{2}\) for \(\xi\in E_{\lambda}\)
  (see Lemma 6.7 of \cite{greenroth}) and \(|\hat\phi|\le1\),
  \[
    \|\,\hat h\|_{\ell^{3}}^{3}
    \;\lesssim\;
    \eta^{6}\!\!\sum_{\xi\in E_{\lambda}}\!|\hat f(\xi)|^{3}
    +\!
    \sum_{\xi\notin E_{\lambda}}\!|\hat f(\xi)|^{3}
    \;\le\;
    \eta^{6}M^{3}+\lambda^{3-q}M^{q}.
  \]
  Hence
  \begin{equation}\label{hbound}
    \|\,\hat h\|_{\ell^{3}}
    \;\lesssim\;
    \eta^{2}M+(\lambda M)^{\frac{3-q}{3}}M^{q/3}.
  \end{equation}
  
  \medskip
  \noindent
  {\bf Step 3: Choose \(\lambda\).}
  Put \(\lambda:=M\eta\).  Then \eqref{gbound}–\eqref{hbound} become
  \begin{align}
    \label{9}
    \|\,\hat g\|_{\ell^{2}}
    &\;\lesssim\;
    \eta^{-\frac{q-2}{2}}\,M,\\[3pt]
    \label{10}
    \|\,\hat h\|_{\ell^{3}}
    &\;\lesssim\;
    \eta^{\frac{3-q}{3}}\,M,
  \end{align}
  where we used \(\eta<1\) and \(q\ge2\) in \eqref{10}.
  
  \medskip
  \noindent
  {\bf Step 4: Choose \(\eta\) and \(q\).}
  Demanding the right-hand side of \eqref{10} \(\le\varepsilon\) forces
  \[
    \eta \;\le\;
    \Bigl(\frac{\varepsilon}{M}\Bigr)^{\frac{3}{3-q}}.
  \]
  Take
  \[
    \eta \;:=\;
    \Bigl(\frac{\varepsilon}{M}\Bigr)^{T},
    \qquad
    T>1.
  \]
  Then \(T\ge\frac{3}{3-q}\), i.e. \(q\le3\bigl(1-\frac{1}{T}\bigr)\).
  
  Next require \eqref{9}\(\le CM\), i.e.
  \(\eta^{-(q-2)/2}\le C\).
  With the chosen \(\eta\) this gives
  \[
    q
    \;\le\;
    2+
    \min_{T>1}
    \Bigl\{
        3\bigl(1-\tfrac{1}{T}\bigr),\;
        \tfrac{2\ln C}{\eta^{-1}\ln(\eta^{-1})}
    \Bigr\}
    \;=\;
    2+
    \min_{T>1}
    \Bigl\{
        3\bigl(1-\tfrac{1}{T}\bigr),\;
        \tfrac{2\ln C}{(\frac{\varepsilon}{M})^{-T}\!
                         \ln\bigl((\frac{\varepsilon}{M})^{-T}\bigr)}
    \Bigr\}.
  \]
  
  Since the right-hand side always exceeds \(2\), we can pick \(q\in(2,3)\)
  satisfying both constraints.  With this, \(q\) and the corresponding
  \(\eta\), inequalities \eqref{9}–\eqref{10} give
  the required bounds on \(\|\hat g\|_{\ell^{2}}\) and
  \(\|\hat h\|_{\ell^{3}}\).
  Observing further that \(g=\phi*f\) is non-negative and
  \(\|g\|_{L^{1}}=\|f\|_{L^{1}}\), while
  \(|\hat h|\le|\hat f|\), all five bullet-point properties are verified.
  \end{proof}
  
%
%
%
%
%
%

 \subsection{3AP counts of $L^2$ functions}\label{s:reglemma}

Define 
\[\norm{f}_{Lip} = \norm{f}_{L^{\infty}} + \sup_{x\neq y} \frac{\abs{f(x)-f(y)}}{\abs{x-y}}\]
and
\[\norm{f}_{U^2} = \paren{ \int \abs{\int f(x)f(x-r)\,dx}^2\,dr}^{1/4}.\]

Given numbers $M$ and $N$, say that $\theta\in\T^d$ is $(M,N)$-irrational if for all $q\in \Z^d$ with $\norm{q}_{\ell^1}\leq M$, $\norm{q\cdot\theta}_{\T}\geq \frac{M}{N}$. An analysis of the proof of Theorem 1.2 of \cite{taogreenregularity} reveals that only $L^2$-bounds on the function $f$ are needed for our purposes, and further that 
the same proof works over $[N]^d$ (provided the availability of an ``inverse theorem'', which is trivially the case for $U^2$),
 so that a slight modification would yield the following. We will, however, give a complete proof (of a slightly more complicated statement, Lemma \ref{thm:reglemma_extended}) below.
\begin{lemma}\label{thm:reglemma}
  Suppose that $f:[N]\ra[0,\infty)$ with $\norm{f}_{L^1}\geq\delta>0$, $f\in L^2$ with $\norm{f}_{L^2}\leq 1$. Let $\mathcal{F}:\N\ra\N$ be an increasing function and let $\eps>0$. Then there exists an $M\in\N$ and a decomposition $f = f_{str}+f_{U^2}+f_{L^2}$ such that
   \[\norm{f_{U^2}}_{U^2}\leq \mathcal{F}(M),\]
   \[\norm{f_{L^2}}_{L^2}\leq \eps,\]
   and
   \[f_{str}(n) = F(n/N,n\bmod q, \theta n)\]
   where
\[
  F:[0,1]\times\Z/q\Z\times\T^d\to[0,1],
\]
$q,d,\|F\|_{Lip} \leq M$, and $\theta\in\T^d$ is $(\mathcal{F}(M),N)$-irrational.
\end{lemma}
%
%
%
 
 We find it convenient to use Herglotz's Theorem in the proof of Corollary \ref{thm:L2count}.
 
 \begin{definition}
  A continuous function $\psi$ on a locally compact abelian group $G$ is \textbf{positive-definite} if for all finite complex-valued sequences $a_n$
  \[ \sum_{i,j} \psi(i-j)a_i\overline{a_j}\geq 0.\]
 \end{definition}

 \begin{theorem}\label{thm:Herglotz}[Herglotz's Theorem]
  Let $f$ be a positive finite Borel measure on a locally compact abelian group. Then $\hat{f}$ is positive-definite on the Pontryagin dual $\hat{G}$ of $G$.
  Conversely, if $\psi$ is a positive-definite function on a locally compact abelian group, then $\psi$ is the Fourier transform of a positive finite Borel measure on the dual group.
 \end{theorem}

  \begin{proof}[Proof of Corollary \ref{thm:L2count}]
   We deal with the case that $X=[N]^d$ first. Let $\eps>0$. Use Lemma \ref{thm:reglemma} 
   to write $f = f_{str}+f_{U^2}+f_{L^2}$, where
   \[\norm{f_{U^2}}_{U^2}\leq \mathcal{F}(M),\]
   \[\norm{f_{L^2}}_{L^2}\leq \eps,\]
   and
   \[f_{str}(n) = F(n/N,n\bmod q, \theta n)\]
   where
\[
  F:[0,1]\times\Z/q\Z\times\T^d\to[0,1],
\]
$q,d,\|F\|_{Lip} \leq M$, and $\theta\in\T^d$ is $(\mathcal{F}(M),N)$-irrational.

 Then by following exactly the proof in \cite{taogreenregularity} of their Theorem 6.1 (Szemeredi's Theorem), replacing their use of the Regularity Lemma for 1-bounded functions by Lemma \ref{thm:reglemma}, one obtains the claim. Now if instead $X=[0,1]^d$, we embed $[0,1]^d\subset [-1,2]^d\hookrightarrow \T^d$, noting that this embedding preserves the $3$AP count (abusing notation, we continue to write $f$ for the resulting function defined on $\T^d$).
 
 Let $\eps>0$. Let $N$ be such that 
 \[\paren{\sum_{\xi\in\Z^d, \abs{\xi}>N/100} \abs{\hat{f}(\xi)}}^{\half}\leq \eps\]
 and write $f_N = \phi_N\ast f$ were $\phi$ is a mollifier satisfying $\widehat{\phi_N}|_{B(0,N/2)}\equiv 1$, $\supp \widehat{\phi_N}\subset B(0,N)$, so that $\norm{f-f_N}_{L^2}\leq\eps$. 
 We have that 
 \[\abs{\Lambda_3(f)-\Lambda_3(f_N)}\lesssim\eps^3\] 
 since
 \eqn{ 
 \abs{\Lambda_3(f)-\Lambda_3(f_N)} = \abs{\sum_{\xi\in\Z^d}\hat{f}(\xi)^2\hat{f}(-2\xi)-\hat{f_N}(\xi)^2\hat{f_N}(-2\xi)}
 \\\leq&
  \abs{\sum_{\abs{\xi}\leq N/2}\hat{f}(\xi)^2\hat{f}(-2\xi)-\hat{f_N}(\xi)^d\hat{f_N}(-2\xi)} + \abs{\sum_{\abs{\xi}> N/2}\hat{f}(\xi)^2\hat{f}(-2\xi)-\hat{f_N}(\xi)^d\hat{f_N}(-2\xi)}
  \\\leq&
  0 + 2\abs{\sum_{\abs{\xi}> N/2}\hat{f}(\xi)^2\hat{f}(-2\xi)}\lesssim\norm{\hat{f}}_{\ell^3(\abs{\xi}>N/2)}^3\leq \norm{\hat{f}}_{\ell^2(\abs{\xi}>N/2)}^3\leq\eps^3.
 }
 
 Let $p\in (10N,20N)$ be a prime and $g$ denote the function on $\paren{p^{-1}\Z_{p}}^d$ given by 
 \[g(x) = \sum_{\xi\in B(0,N)\subset\paren{\Z/p\Z}^d}\hat{f_N}(\xi)e^{2\pi i \xi\cdot x}.\]
 If we embed $\paren{p^{-1}\Z_p}^d$ within $\T^d$ in the natural way, $g(x) = f_N(x)$. This identification may be sidestepped, however, note that
 \begin{enumerate}
  \item $\hat{g}$ is positive-definite since by Herglotz's Theorem $\hat{f_N}$ is,
  \item thus $g\geq 0$ again by Herglotz's Theorem,
  \item $\norm{g}_{L^1} = \hat{g}(0)=\hat{f_N}(0)=\norm{f}_{L^1}$, 
  \item $\norm{g}_{L^2}=\norm{f_N}_{L^2}\leq\norm{f}_{L^2}\leq 1$ by Plancherel's theorem,
  \item $\Lambda_3(g)=\Lambda_3(f_N)$ by the identity
 \[\Lambda_3(g) = \sum_{\xi}\hat{g}(\xi)^2\hat{g}(-2\xi)\]
 applied first to $g$, then to $f_N$, and using that $\hat{g}=\hat{f_N}$ under the identification
 \[\Z_p = \set{-\frac{p-1}{2},\dots,0,\dots,\frac{p-1}{2}}\subset \Z.\]
 \end{enumerate}
 
 By the conclusion of this lemma over $[p]^d$, we therefore obtain that
 \[\Lambda_3(f)+O(\eps^3)=\Lambda_3(f_N)=\Lambda_3(g)\geq c^{(2)}\paren{\norm{f}_{L^1}}.\]
  \end{proof}

 \subsubsection{Regularity Lemmas}

First, we'll state a standard \((L^{\infty},U^{2})\) regularity lemma.

 \begin{lemma}[Abelian \((L^{\infty},U^{2})\) regularity lemma {\cite[Thm.\;5]{Eberhard16}}]
\label{lem:bounded-reg}
Let $f:[N]^{d}\to[0,1]$ and let $F:\mathbb N\to\mathbb N$ be any growth function.
For every $\varepsilon>0$ there exists $M\ll_{\,\varepsilon,F} 1$ and a decomposition
\[
f \;=\; f_{\mathrm{str}} \;+\; f_{\mathrm{sml}} \;+\; f_{\mathrm{unf}}
\]
with
\begin{enumerate}
\item \(f_{\mathrm{str}}(n)=
      F\bigl(n/N,\;n\bmod q,\;\theta n\bigr)\) where  
      \(F:[0,1]\times\mathbb Z/q\mathbb Z\times\mathbb T^{d}\to[0,1]\),  
      \(q,d,\lVert F\rVert_{\mathrm{Lip}}\le M\), and
      \(\theta\in\mathbb T^{d}\);
\item \(\lVert f_{\mathrm{sml}}\rVert_{L^{2}([N]^{d})}\le\varepsilon\);
\item \(\lVert f_{\mathrm{unf}}\rVert_{U^{2}([N]^{d})}\le 1/F(M)\);
\item \(f_{\mathrm{str}},\,f_{\mathrm{str}}+f_{\mathrm{sml}}\in[0,1]\).
\end{enumerate}
\end{lemma}

Extending this to the $(L^{2}, U^{2})$ arithmetic regularity on $\,[N]^{d}$ form we need is straightforward:
(Below we'll treat $N$ as prime, embedding $[N]$ into $[N']$ for some prime $N'$ if necessary, at the cost of some constants which we'll suppress.)
\begin{lemma}
\label{thm:reglemma_extended}
Fix an integer $d\ge 1$. Let 
\[
f:[N]^{d}\longrightarrow[0,\infty)
\qquad\text{with}\qquad 
\|f\|_{L^{1}([N]^{d})}\;\ge\;\delta
\quad\text{and}\quad
\|f\|_{L^{2}([N]^{d})}\;\le\;1 .
\]

For every growth function $\mathcal F:\N\to\N$ and every $\varepsilon>0$
there exists an integer
\[
M=M(d,\delta,\varepsilon,\mathcal F)\;\ll_{\delta,\varepsilon,\mathcal F,d}\!1
\]
and a decomposition
\[
f
\;=\;
f_{\mathrm{str}}
\;+\;
f_{U^{2}}
\;+\;
f_{L^{2}}
\]
with the following properties:

\begin{enumerate}
\item  (\emph{structured part})  
       there are integers $q,d'\le M$, a vector $\theta\in\T^{d'}$ that is
       \(\bigl(\mathcal F(M),N\bigr)\)-irrational, and a Lipschitz function
       \[
         F:[0,1]\times\Z/q\Z\times\T^{d'}\;\longrightarrow\;[0,1],
         \qquad \|F\|_{\mathrm{Lip}}\le M ,
       \]
       such that for every $n\in[N]^{d}$
       \[
         f_{\mathrm{str}}(n)
         \;=\;
         F\!\bigl(n/N,\;n\bmod q,\;\theta\!\cdot n\bigr);
       \]
\item  (\emph{$U^{2}$–uniform part})
       \[
         \bigl\|f_{U^{2}}\bigr\|_{U^{2}([N]^{d})}
         \;\le\;
         \frac{1}{\mathcal F(M)} ;
       \]
\item  (\emph{$L^{2}$-small part})
       \[
         \bigl\|f_{L^{2}}\bigr\|_{L^{2}([N]^{d})}
         \;\le\;
         \varepsilon ;
       \]
\item  Pointwise bounds:
       \(
       f_{\mathrm{str}}(n),\;
       f_{\mathrm{str}}(n)+f_{L^{2}}(n)\in[0,1]
       \) for all $n\in[N]^{d}$.
\end{enumerate}
\end{lemma}

\begin{proof}[Proof of Lemma~\ref{thm:reglemma_extended}]
Throughout write $X=[N]^{d}$ with normalised counting measure
$m_{X}(A)=|A|/N^{d}$ and Fourier characters
$e_{N}(x\!\cdot\!\xi):=e^{2\pi i x\cdot\xi/N}$,
$\xi\in\Z_{N}^{d}$.

\medskip
\noindent
\textbf{Step 1.  Truncation.}
Fix $\varepsilon>0$ and a growth function
$\mathcal F_0\colon\N\to\N$.  Choose a cut-off
\[
K:=K(\varepsilon):=\frac{3}{\varepsilon}\,,
\]
and decompose $f$ into
\[
f_{\le}(x):=\min\bigl(f(x),K\bigr),\qquad
f_{>}(x):=f(x)-f_{\le}(x).
\]
By Chebyshev,
\(
m_{X}(f>K)\le \|f\|_{2}^{2}/K^{2}\le\varepsilon^{2}/9,
\)
hence
\[
\|f_{>}\|_{2}^{2}\le K^{2}\,m_{X}(f>K)\le\varepsilon^{2}/3
\quad\Longrightarrow\quad
\|f_{>}\|_{2}\le\varepsilon/\sqrt{3}.
\tag{$\dagger$}
\]

\medskip
\noindent
\textbf{Step 2.  Apply the bounded regularity lemma.}
The function $f_{\le}$ satisfies
$0\le f_{\le}\le K$ and $\|f_{\le}\|_{1}\ge\delta$.
Invoke the \((L^{\infty},U^{2})\) result, Lemma \ref{thm:reglemma_extended} 
with bound $M:=K$ and error parameter $\varepsilon/3$:
there is an
\[
M_{0}=M_{0}(\varepsilon,\mathcal F_0,K,\delta)\in\N
\]
and a decomposition
\[
f_{\le}=g_{\mathrm{str}}+g_{U^{2}}+g_{L^{2}}
\]
such that
\begin{align}
&\|g_{U^{2}}\|_{U^{2}}\;\le\;\mathcal F_0(M_{0}),\tag{2.1}\\
&\|g_{L^{2}}\|_{2}\;\le\;\varepsilon/3,\tag{2.2}\\
&g_{\mathrm{str}}(n)=
  F\!\bigl(n/N,\;n\bmod q,\;\theta n\bigr)
     \quad
  \text{with } q,d,\|F\|_{\mathrm{Lip}}\le M_{0},\;
  \theta\text{ $(\mathcal F_0(M_{0}),N)$-irrational.}\tag{2.3}
\end{align}

\medskip
\noindent
\textbf{Step 3.  Assemble the final decomposition.}
Put
\[
f_{\mathrm{str}}:=g_{\mathrm{str}},\qquad
f_{U^{2}}:=g_{U^{2}},\qquad
f_{L^{2}}:=g_{L^{2}}+f_{>}.
\]
Then clearly $f=f_{\mathrm{str}}+f_{U^{2}}+f_{L^{2}}$.
The structural description (2.3) furnishes the desired
complexity bounds with
\(
M:=M_{0}.
\)

\medskip
\noindent
\textbf{Step 4.  Norm estimates.}
The $U^{2}$ bound is exactly (2.1).  For the $L^{2}$ tail,
combine (2.2) with $(\dagger)$:
\[
\|f_{L^{2}}\|_{2}\;\le\;\|g_{L^{2}}\|_{2}+\|f_{>}\|_{2}
\;\le\;\varepsilon/3+\varepsilon/\sqrt{3}
\;<\;\varepsilon.
\]

\medskip
\noindent
\textbf{Step 5.  Conclusion.}
All requirements in Lemma~\ref{thm:reglemma} are met with
\[
M\;=\;M_{0}\bigl(\varepsilon,\mathcal F_0,\tfrac{3}{\varepsilon},\delta\bigr).
\qedhere
\]
\end{proof}



\renewcommand{\su}[2]{\oldsu{#1}{#2}}



\section{Fractal results}\label{s:fractalresults}

\renewcommand{\si}[1]{\sigma\paren{#1}}
\newcommand{\Si}[1]{\varsigma\paren{#1}}
\renewcommand{\di}[1]{\delta\paren{#1}}

\renewcommand{\su}[2]{\int #2\,d#1}
\renewcommand{\ro}{\rho}
%

Let $\mu$ be a Radon measure on $\T^d$ or on $[0,1]^d$. Whenever this integral exists, define 
 \[\si{\mu}(\ro):= \int \int \widehat{\mu}(2\eta)\widehat{\mu}(\eta-\ro\theta)\overline{\widehat{\mu}(\eta+\ro\theta)}\,d\sigma(\theta)\,d\eta,\]
where the $\sigma(\theta)$ inside the integral is the Lebesgue measure on the unit sphere. Also, define
 \[\Si{\mu}(\rho) =  \int \int \abs{\widehat{\mu}(2\eta)\widehat{\mu}(\eta-\ro\theta)\overline{\widehat{\mu}(\eta+\ro\theta)}}\,d\sigma(\theta)\,d\eta,\]
 and
$$\Lambda_3(f):= \int \int f(x)f(x-r)f(x-2r)dxdr,$$
when these integrals exist. In what follows, we will freely use $(\phi_n)_{n\in \mathbb N}$ as an approximate identity with $\supp \widehat{\phi_n}\subset B(0,2^n).$
Finally, let $\di{\mu}$ be the measure on $\R$ defined as
\[\int g(r)\,d\di{\mu}(r) := \lim_{n\ra\infty} \int g(\abs{u})\prod_{i\in[3]}\phi_n\ast\mu(x-iu)\,dx\,du\]
 provided that the right hand side above defines a continuous linear functional on $C(\R)$ (as it does, e.g., if $\mu$ has a density function $f\in C(\R^d)$).

 We note that the measures on $[0,1]^d$ may be identified with (a subset of) those on $\T^d$ via the embedding given by first mapping $[0,1]^d$ to the unit cube in $\R^d$, followed by identifying the cube $[-2,2]^d$ with $\T^d$ in the natural way. This identification preserves the energies $I_{\alpha}(\mu)$ and the $L^p$ norms $\norm{\widehat{\mu}}_{L^p}$ up to multiplicative constants, and also preserves the $3$AP counts as measured by $\Lambda_3(\mu)$ and the (up to rescaling) set of lengths of $3$APs as measured by $\di{\mu}$ when the definition of $\di{\mu}$ is extended to $\T^d$ in the obvious way. In this section we will perform this identification without further comment.

\subsection{Proof of Theorem \ref{thm:MAIN}}
The proof of Theorem \ref{thm:MAIN} follows from Lemma \ref{thm:distancemeasuresupport} (Support), Lemma \ref{thm:distancemeasuremass} (Mass), and Lemma \ref{thm:distancemeasureLInftyBound} (Non-singularity), below.

\subsubsection{Support Lemma}

 We state the following without proof.
 \begin{lemma}\label{thm:distancemeasuresupport}[Support lemma.]
 Let $\mu$ be a probability measure on $\T^d$ and suppose that $\di{\mu}$ exists and has finite total mass.
 Then
  \[\supp(\di{\mu})\subset \set{|u| : x, x+u, x+2u \in \supp(\mu)}.\]
 \end{lemma}
\subsubsection{Mass Lemma}
In order to prove the Mass Lemma, we need the following technical results.

\begin{lemma}\label{thm:boundmollifiedmeasurefromenergy}
 Let $\mu$ be a probability measure on $[0,1]^d$ satisfying $I_{\alpha}(\mu) < \infty$, and $(\phi_n)_{n\in\N}$ an approximate identity with $\supp\widehat{\phi_n}\in B(0,2^n)$. Then for all $n$
 \[\norm{\phi_n\ast\mu}_{L^{\infty}}\leq \sqrt{I_{\alpha}(\mu)} 2^{n(d-\alpha)/2}.\]
\end{lemma}
\begin{proof}
 Since $\phi_n\ast\mu$ has an $\ell^1$ Fourier transform, we have by the triangle inequality and Cauchy-Schwarz
 \eqn{ 
 \abs{\phi_n\ast\mu(x)} = \abs{\sum \widehat{\phi_n}(\xi)\widehat{\mu}(\xi) e^{2\pi i \xi\cdot x}}
 \\\leq& 
 \sum_{\abs{\xi}\leq 2^n} \paren{\abs{\widehat{\mu}(\xi)}\abs{\xi}^{-(d-\alpha)/2}}\paren{\abs{\xi}^{(d-\alpha)/2}}
 \\\leq&
 \sqrt{I_{\alpha}(\mu)} 2^{n(d-\alpha)/2}.
 }
\end{proof}

\begin{lemma}\label{thm:decayToLq}
  Suppose that the measure $\mu$ on $\T^d$ satisfies \eqref{itm:b} for $\beta>0$. Then 
  \[\norm{\widehat{\mu}}_{\ell^q}^{q}\leq I_{\alpha}(\mu)\sup_{\xi}\abs{\widehat{\mu}(\xi)}^{q-2}\abs{\xi}^{d-\alpha}.\]
  In particular, $\widehat{\mu}\in\ell^q$  for $q = 2+2\frac{d-\alpha}{\beta}=2\frac{\beta+d-\alpha}{\beta}$.
 \end{lemma}
 \begin{proof}
  We have
  \begin{align*}&
   \sum_{\xi\in\Z^d} \abs{\widehat{\mu}}^q
   \\=&
   \sum_{\xi\in\Z^d} \paren{\abs{\widehat{\mu}}^2\abs{\xi}^{-(d-\alpha)}}\paren{\abs{\widehat{\mu}}^{q-2}\abs{\xi}^{d-\alpha}}
   \\\leq&
   I_{\alpha}(\mu)\sup_{\xi}\abs{\widehat{\mu}(\xi)}^{q-2}\abs{\xi}^{d-\alpha}
  \end{align*}
 \end{proof}
 
 From here forward, we define $J := J_{(d-2)/2}$, the Bessel function of order $(d-2)/2$. We record some well-known Bessel function estimates in the following lemma, which may be found in \cite{MattilaFalconer} (2.1), (2.2), and the first displayed formula on page 216.
 
 \begin{lemma}\label{thm:bessel}
 Consider the Bessel function $J=J_{(d-2)/2}.$
 \begin{enumerate}
  \item For all $\xi\in\R$
  \[\abs{\widehat{J}(\xi)}\leq \abs{\xi}^{-1/2}.\]
  \item For all $\xi\in\R$
  \[\abs{\widehat{J}(\xi)}\leq \abs{\xi}^{(d-2)/2}.\]
  \item 
  There is a constant $C$ so that the following holds. Let $g\in C(\R)$ and $G\in C(\R^d)$ satisfy $G(u) = g(\abs{u})$. Then
  \[\widehat{G}(\xi) = C \abs{\xi}^{-(d-2)/2}\int_0^{\infty} r^{d/2}J(r\abs{\xi}) g(r)\,dr.\]
 \end{enumerate}

 \end{lemma}

  \begin{lemma}\label{thm:fourpolarrep}
  For any $f,g\in C(\R^d)$ with $\hat{g}\in L^1$ and $\supp(g)\subset[0,\infty)$.
  \[\int g(r)\,d\di{f}(r) = O(1)\int g(r) r^{d/2} \su{\ro}{ \ro^{d/2}J(r\ro)\si{f}(\ro)}\,dr.\]
  Further, if $\widehat{f}\in L^1$,
  \[\di{f}(r) =  O(1) r^{d/2} \su{\ro}{\ro^{d/2}J(r\ro)\si{f}(\ro)}.\]
 \end{lemma}
 \let\oldsu\su
 \renewcommand{\su}[2]{\int #2\,d\eta\,d\xi}
 \begin{proof}
  Let $G(u) = g(\abs{u})$. 
  Then using Lemma \ref{thm:bessel} and writing $\xi = s\theta$ in polar coordinates
  \begin{align*}&
  \int g\,d\di{f} =  \int G(u)\prod_{i\in[3]}f(x-ir)\,dx\,dr
   \\=&
    \su{\xi,\eta}{\widehat{G}(\xi)\hat{f}(\eta+\xi)\overline{\hat{f}(\eta-\xi)}\hat{f}(2\eta)}
    \\=&
    r^{d/2}\int_{0}^{\infty}\!\!\int_{\mathbb S^{d-1}}
        s^{d/2}J(rs)\,\varsigma_f(s)\,d\sigma(\theta)\,ds
   \\=&
    O(1)\su{\xi,\eta}{ \abs{\xi}^{-(d-2)/2}\int_0^{\infty} r^{d/2}J(r\abs{\xi}) g(r)\,dr\,\hat{f}(\eta+\xi)\overline{\hat{f}(\eta-\xi)}\hat{f}(2\eta)}
   \\=&
    O(1)\int_0^{\infty}	 g(r) \abs{r}^{d/2}\su{\xi,\eta}{ \abs{\xi}^{-(d-2)/2}J(r\abs{\xi})\hat{f}(\eta+\xi)\overline{\hat{f}(\eta-\xi)}\hat{f}(2\eta)}\,dr
   \\=&
    O(1)\int_0^{\infty} g(r) \abs{r}^{d/2}\iiint s^{d-1} s^{-(d-2)/2}J(rs)\hat{f}(\eta+s\theta)\overline{\hat{f}(\eta-s\theta)}\hat{f}(2\eta)\,d\eta\,d\sigma(\theta)\,ds\,dr
   \\=&
    O(1)\int_0^{\infty} g(r) \abs{r}^{d/2}\int s^{d/2} J(rs)\si{f}(s)\,ds\,dr.
  \end{align*}
  Under the hypothesis that $\widehat{f}\in L^1$, it is immediate that $\si{f}$ is bounded, and thus that the above gives the desired pointwise estimate.
 \end{proof}

 \begin{lemma}\label{thm:distancemeasuremass}[Mass lemma.]
  Let $\mu$ be a probability measure satisfying \eqref{itm:a} and \eqref{itm:b} and suppose that the measure $\di{\mu}$ exists. Then if
  \[ (2\beta + d-\alpha)/\beta < 3\]
  and $\alpha$ is sufficiently close to $d$ depending on the quantities $C_F,\beta$, and $I_{\alpha}(\mu)$, the measure $\di{\mu}$ has positive mass.
 \end{lemma}
\begin{proof}
 Let $(\phi_n)$ be an approximate identity satisfying the hypotheses of Lemma \ref{thm:boundmollifiedmeasurefromenergy}, so that for any $n\in\N$
 \[\norm{\phi_n\ast\mu}_{L^{\infty}}\leq \sqrt{I_{\alpha}(\mu)}2^{n(d-\alpha)}\]
 and $\norm{\phi_n\ast\mu}_{L^1}=1$.
 
 Then by Varnavides' Theorem  there is a monotonic decreasing $c = c\paren{\norm{\phi_n\ast\mu}_{L^{\infty}}}\geq c(\sqrt{I_{\alpha}(\mu)}2^{n(d-\alpha)})$ (so that $c$ is in particular independent of any other properties of the measure $\mu$), such that
 \[\Lambda_3(\phi_n\ast\mu)\geq c.\]
 Now fix $N\in\N$. Consider the difference 
 \eqn{ 
 \abs{\norm{\di{\mu}} - \Lambda_3(\phi_N\ast\mu)} = \lim_{n\ra\infty}\abs{\Lambda_3(\phi_n\ast\mu) - \Lambda_3(\phi_N\ast\mu)}
 \\\leq&
 \sum_{n\geq N} \abs{\Lambda_3(\phi_{n+1}\ast\mu) - \Lambda_3(\phi_n\ast\mu)}
 }
 Let 
 \[F_n =\set{\paren{f,f,f-g},\paren{f,f-g,g},\paren{f-g,g,g}}\]
 where $f=\phi_{n+1}\ast\mu$ and $g=\phi_{n}\ast\mu$, so that each triple in $F_n$ has exactly one entry equal to $(\phi_{n+1}-\phi_n)\ast\mu$.
Using the identity $\Lambda_3(\phi_{n+1}\ast\mu)-\Lambda_3(\phi_n\ast\mu) = \sum_{(f_0,f_1,f_2)\in F_n}\Lambda_3(f_0,f_1,f_2)$, which follows from the identity $\prod a_i - \prod b_i = \sum_j \paren{\prod_{i<j} a_i}\paren{a_j-b_j}\paren{\prod_{i>j} b_j}$ applied to the integrands, we bound
 \eqn{ 
 \abs{\Lambda_3(\phi_{n+1}\ast\mu) - \Lambda_3(\phi_{n}\ast\mu)}
 \\\leq&
 12 \sup_{\vec{f}\in F_n} \abs{\Lambda_3(f_0,f_1,f_2)}.
 }
 Taking Fourier transform, applying H\" older's inequality and using the support properties of $\paren{\widehat{\phi_{n+1}}-\widehat{\phi_n}}\widehat{\mu}$ (which is equal to one of the $f_i$), we have 
 \eqn{ 
 \abs{\Lambda_3(f_0,f_1,f_2)} = \abs{\sum_{\eta\in\Z^d} \widehat{f_0}(\eta)\overline{\widehat{f_1}(\eta)}\widehat{f_2}(2\eta)}
 \\\leq&
 \prod_{i=0}^2\norm{\hat{f_i}}_{\ell^{3}\paren{B(0,2^{n+1})\setminus B(0,2^{n-1})}}\leq \prod_{i=0}^{2}\paren{I_{\alpha}(\mu)^{1/3}\norm{\widehat{\mu}(\xi) \abs{\xi}^{d-\alpha}}_{\ell^{\infty}(\abs{\xi}\approx 2^n)}^{\frac{1}{3}}}
 }
 where in the last line we have used Lemma \ref{thm:decayToLq}. Together with \eqref{itm:b} this gives
  \[\abs{\Lambda_3(f_0,f_1,f_2)} \leq C_F I_{\alpha}(\mu) 2^{-\frac{\beta-2(d-\alpha)}{2}n}.\]
  
  This yields that
  \begin{align}&\label{1818}\abs{\norm{\di{\mu}} - \Lambda_3(\phi_N\ast\mu)} \leq 12 C_F I_{\alpha}(\mu) 2^{-(\beta - (d-\alpha))N/2}/\paren{1-2^{-(\beta - 2(d-\alpha))/2}}.\end{align}
  
  Fix a lower bound $\alpha_0$ for $\alpha$ and an upper bound $I_{\alpha_0}$ for $I_{\alpha}(\mu)$ such that the constant $c_0:= c(\sqrt{I_{\alpha_0}}2^{n(d-\alpha_0)})$ coming from Varnavides Theorem is positive. Fix also a lower bound $\beta_0$ for $\beta$. We may suppose that $\alpha_0$ and $\beta_0$ satisfy the condition $ (2\beta_0 + d-\alpha_0)/\beta_0 < 3$. 
  
  Let $N$ be large enough that the expression
  \[12 C_F I_{\alpha_0} 2^{-(\beta_0 - (d-\alpha_0))N/2}/\paren{1-2^{-(\beta_0 - 2(d-\alpha_0))/2}},\]
  which upper-bounds the right-hand-side of \eqref{1818} is strictly less than $c_0$, which is possible since $\beta_0 > 2(d-\alpha_0)$. Choose $\alpha>\alpha_0$ sufficiently close to $d$ that the Varnavides bound $c$ satisfies
  \[ c(\sqrt{I_{\alpha}(\mu)}2^{n(d-\alpha)}) \geq c_0,\]
  which is possible since $c$ is decreasing in its argument.
  Putting these together, we obtain from the reverse triangle inequality that
  \[\norm{\di{\mu}}=\abs{\Lambda_3(\phi_N\ast\mu)- \paren{\Lambda_3(\phi_N\ast\mu)-\norm{\di{\mu}}}}\geq \Lambda_3(\phi_N\ast\mu)-\abs{\norm{\di{\mu}} - \Lambda_3(\phi_N\ast\mu)}> c_0-c_0 = 0\]
  showing that $\norm{\di{\mu}}>0$. 
  
\end{proof}
\begin{remark}
 From the proof we see that we may take $\alpha>d-\frac{\beta}{2}$ provided that for some $N\in\N$ and $c=c(\cdot)$ 
  the lower bound function $c$ coming from Varnavides Theorem \ref{thm:varnThm} satisfies the strict inequality
 \[ c\paren{ \sqrt{I_{\alpha_0}}C_F2^{N(d-\alpha_0)}} > \frac{12C_F I_{\alpha_0}}{\paren{1-2^{-(\beta_0 - 2(d-\alpha_0))/2}}} 2^{-(\beta_0 - (d-\alpha_0))N/2}\]
for data satisfying the following bounds
 \eqn{\alpha\geq \alpha_0
 ,\\&
 I_{\alpha}(\mu)\leq I_{\alpha_0}
 ,\\&
 \sup_{\xi\in\Z^d} \frac{\abs{\widehat{\mu}(\xi)}^2}{\abs{\xi}^{\beta}}\leq C_F.}
\end{remark}

\subsubsection{Non-singularity Lemma}

 We first establish existence of the measure $\di{\mu}$.
 
\begin{lemma}\label{thm:existence}[Existence lemma.]
 Let $\mu$ be a Radon probability measure on $\T^d$. Suppose that $\widehat{\mu}\in \ell^q$ for some $q\leq 3$. 
 Then the measure
 $\di{\mu}$ exists. In particular, if $\mu$ satisfies \eqref{itm:a} and \eqref{itm:b} then this holds when $q=2\paren{\frac{d+\beta-\alpha}{\beta}}$ provided this value is less than or equal to $3$.
\end{lemma}
\begin{proof}
 Let $(\phi_n)$ be an approximate identity. By the Riesz representation theorem, it is sufficient 
 to show that the limit of the triple-convolution integrals converge for every test function $g$ in a dense subclass of $C(\R)$ 
 in order to define a finite Radon measure; this defines $\delta(\mu)$.
 
 Thus, that $\di{\mu}$ exists follows from a number of observations: the absolute convergence of the sums
 \[ \sum_{\xi\in\Z^d}\abs{\widehat{\phi_n\ast\mu}(\xi)}^2\abs{\widehat{\phi_n\ast\mu}(-2\xi)}\] and more generally for any fixed $\eta_0,\eta_1$ of
 \[  \sum_{\xi\in\Z^d}\abs{\widehat{\phi_n\ast\mu}(\xi-\eta_0)}\abs{\widehat{\phi_n\ast\mu}(\xi-\eta_1)}\abs{\widehat{\phi_n\ast\mu}(-2\xi)},\]
 the observation that this convergence guarantees the existence of the limits
 \eq{\label{repby1}\lim_{n\ra\infty} \iint f(x,r)\prod_{i=0}^2\phi_n\ast\mu(x-ir)\,dx\,dr}
 for any trigonometric monomial $f(x,r)=e^{2\pi i (x\cdot(\eta_0-\eta_1)+r\cdot\eta_1)}$, so for any trigonometric polynomial $f$, hence for any continuous function $f$, and the existence of a measure $\cap^3\mu$ on $\T^d\times\T^d$ represented by the action on $C(\T^d)$ given by \eqref{repby1}, and finally by the observation that for a function $g$ on $\R$, 
 \[\int_0^{\infty} g(r)\,\di{\mu}(r) = \int g(\abs{u})\,d\cap^3\mu(x,u).\]
\end{proof}

The following technical estimates will be used in proving the Non-singularity lemma.

 \begin{lemma}\label{thm:mattilaLInfintymain}
 
  If $\rho>d(1-1/q)$ then
    \[\int r^{d-1-\rho}\Si{\mu}(r)\,dr\lesssim \norm{\widehat{\mu}}_{L^q}^3.\]
 In particular, if $\mu$ sastisfies \eqref{itm:b} and $\rho>(4d-3\beta)/2$ then
  \[\int r^{d-1-\rho}\Si{\mu}(r)\,dr\lesssim 1.\]
 \end{lemma}
 \begin{proof}
  Suppose first that $\widehat{\mu}\in L^1$. Converting from polar coordinates we have
 \begin{align*}&\notag{}
 \int r^{d-1-\rho}\Si{\mu}(r)\,dr
 \\\leq& 
 \int\int \abs{\xi}^{-\rho} \abs{\widehat{\mu}(\xi+\eta)\overline{\widehat{\mu}(\xi-\eta)}\widehat{\mu}(\eta)}\,d\xi\,d\eta.
 \end{align*}
 Using the Brascamp-Lieb inequality we have that this is bounded above by
 \begin{align*}&
\leq \norm{\abs{\xi}^{-\rho}}_{L^p} \norm{\widehat{\mu}}_{L^q}^3.
 \end{align*}
The first statement of the result then follows upon taking $\rho>d(1-1/q) = d/q'$ and applying a limiting argument to the measures $\phi_n\ast\mu$.
 Using the Fourier decay of $\mu$, if $1/p+3/q = 2$, $q>2d/\beta$, and $p>d/\rho$ this is finite precisely when $\rho> (4d-3\beta)/2$, proving the second statement.
 \end{proof}

\begin{prop}\label{thm:distancemeasureLInftyBound-function}
 Suppose 
$f\in C(\R^d)$ is a positive function with $\hat{f}\in L^1$. Fix $\rho\in(0,d-1)$. Then 
 \[\di{f}(s)\lesssim s^{d-1-\rho} \int_0^{\infty} r^{d-1-\rho}\abs{\si{f}(r)}\,dr.\]
\end{prop}
\begin{proof}
 By Lemma \ref{thm:fourpolarrep} we have
 \[\di{f}(r) =  O(1) r^{d/2} \int_0^{\infty} s^{d/2}J(rs)\si{f}(s)\,ds.\]
 Split this integral as  
  \[\di{f}(r) \lesssim r^{d/2} \paren{\int_0^{1/r} + \int_{1/r}^{\infty}}s^{d/2}J(rs)\abs{\si{f}(s)}\,ds =: \di{f}_1(r)+\di{f}_2(r).\]
  Then using Lemma \ref{thm:bessel}
 \eqn{
  \di{f}_1(r) 
  \\\lesssim&
  r^{d/2} \int_0^{1/r}(sr)^{(d-2)/2} r^{d/2}\abs{\si{f}(s)}\,ds
  \\=&
  r^{d-1}\int_0^{1/r} s^{d-1}\abs{\si{f}(s)}\,ds
  \\=&
    r^{d-1}\int_0^{1/r} s^{\rho} s^{d-1-\rho}\abs{\si{f}(s)}\,ds
  \\\leq&
    r^{d-1-\rho}\int_0^{1/r} s^{d-1-\rho}\abs{\si{f}(s)}\,ds.
  }
%
Similarly,
\eqn{
  \di{f}_2(r) 
  \\\lesssim& 
  r^{d/2} \int_{1/r}^{\infty}(sr)^{-1/2} r^{d/2}\abs{\si{f}(s)}\,ds
  \\=&
   \int_{1/r}^{\infty}(sr)^{(d-1)/2} \abs{\si{f}(s)}\,ds
  \\=&
   \int_{1/r}^{\infty} r^{(d-1)/2} s^{\rho-(d-1)/2} s^{d-1-\rho}\abs{\si{f}(s)}\,ds
   \\\leq&
   \int_{1/r}^{\infty} r^{d-1} s^{\rho} s^{d-1-\rho}\abs{\si{f}(s)}\,ds
   \\\leq&
   r^{d-1-\rho}\int_{1/r}^{\infty}  (rs)^{\rho} s^{d-1-\rho}\abs{\si{f}(s)}\,ds
   \\\leq&
   r^{d-1-\rho}\int_{1/r}^{\infty} s^{d-1-\rho}\abs{\si{f}(s)}\,ds.  
}

\end{proof}

\begin{lemma}\label{thm:distancemeasureLInftyBound}[Non-singularity lemma.]
 Let $\mu$ be a Radon probability measure on $\R^d$. Suppose that either $\mu$ satisfies \eqref{itm:b} or that $\widehat{\mu}\in L^q$. Let $d-1>(4d-3\beta)/2$ or $d-1>d(1-1/q)$, respectively. Then the measure
 $\di{\mu}$ exists, and further $\di{\mu}\in L^{\infty}$.
\end{lemma}
\begin{proof}

 Existence follows from Lemma \ref{thm:existence}.
Let $f_n = \phi_n\ast\mu$, and let $\rho$ satisfy either $d-1>\rho>(4d-3\beta)/2$ or $d-1>\rho>d(1-1/q)$ as appropriate. By Lemma \ref{thm:mattilaLInfintymain} and the hypotheses on $\mu$ and $\rho$,
 \[\int_0^{\infty} r^{d-1-\rho}\abs{\si{f_n}(r)}\,dr < \infty\]
 uniformly in $n$, so by Lemma \ref{thm:distancemeasureLInftyBound-function} 
 \[\sup_{n,s}\di{f_n}(s) <\infty.\]
 
 Noting also that for functions $g$ and $f$ with $\hat{g},\hat{f}\in L^1$ 
 \[\int g(s)\,\di{f}(s) = \iint g(|r|)\prod_{i=0}^2f(x-ir)\,dx\,dr,\]
 we have by the definition of the measure $\di{\mu}$ that for any $g\in L^1(\R)$ with $\hat{g}\in L^1$
 \[\int_0^{\infty} g\, \di{\mu} = \lim_{n\ra\infty}\int  g(|r|) \prod_{i=0}^2f_n(x-ir)\,dx\,dr  = \lim_{n\ra\infty} \int g(s) \delta\paren{f_{n}}(s)\,ds\lesssim\int \abs{g(s)}\,ds = \norm{g}_{L^1}\]
 and so $\di{\mu}\in L^{\infty}$.
\end{proof}

We also include here the following weaker result, valid when $d=1$.

\begin{lemma}\label{thm:onedimensionalnonsingularity}
  Let $\mu$ be a Radon probability measure on $[0,1]^{d}$ and assume
  $\widehat{\mu}\in L^{q}(\mathbb R^{d})$ for some $2<q\le 3$.
  Then the $3$–AP length measure $\nu=\delta(\mu)$ exists and
  
  \[
    \mathcal H^{s}\!\bigl(\triangle^{3}(\mu)\bigr)
    \;\;\gtrsim\;\;
    \|\nu\|
  \qquad
    \text{for every }s\le 1\text{ with }
    s<\frac12+\frac{d\bigl(6-2q\bigr)}{2q}.
  \]
  In particular, if\/ $\mu$ satisfies \eqref{itm:a}–\eqref{itm:b} then the
  conclusion holds with $q=2\bigl(d+\beta-\alpha\bigr)/\beta$.
\end{lemma}

\begin{proof}
\textbf{Step 1 Existence.}
Lemma \ref{thm:existence} applies directly, so $\nu=\delta(\mu)$ is a
finite Radon measure.

\medskip
\noindent
\textbf{Step 2 Annulus and its Fourier decay.}
Fix $a\in\R^{d}$, radii $0<R<R+\varepsilon\le2$, and let
$\chi:=1_{B(a,R+\varepsilon)}-1_{B(a,R)}$.
Falconer’s estimate (Lemma 2.1 in \cite{falconer}) gives, for any
$s\in(0,1)$,
\begin{equation}\label{eq:Bessel-annulus}
  |\widehat{\chi}(\xi)|
  \;\;\lesssim\;\;
  \min\!\Bigl\{\!
    \varepsilon R^{d-1},\;
    R^{\frac{d-1}{2}}\,\varepsilon^{\,s}\,
    |\xi|^{-\frac{d-1}{2}-(1-s)}
  \Bigr\}.
\end{equation}

\medskip
\noindent
\textbf{Step 3 Fourier representation of the triple integral.}
Let $\phi_{n}$ be any approximate identity and set
\[
  F_{0,n}(x,r)=\chi(r)\,(\phi_{n}\ast\mu)(x),
  \qquad
  F_{1,n}(x,r)=(\phi_{n}\ast\mu)(x-r)\,(\phi_{n}\ast\mu)(x-2r).
\]
Plancherel in $x$ and then in $r$ yields
\[
  \int\chi\,d\nu
  \;=\;
  \lim_{n\to\infty}\!
    \iint F_{0,n}(x,r)\,F_{1,n}(x,r)\,dx\,dr
  \;=\;
  \lim_{n\to\infty}\!
    \sum_{\xi,\eta\in\mathbb Z^{d}}
      \widehat{\chi}(\xi-2\eta)\,
      \widehat{\phi_{n}\!\ast\!\mu}(\eta+\xi)\,
      \widehat{\phi_{n}\!\ast\!\mu}(\xi)\,
      \widehat{\phi_{n}\!\ast\!\mu}(\eta).
\]

\medskip
\noindent
\textbf{Step 4 Brascamp-Lieb / H\:older bound.}
With $1/p+3/q=2$ (hence $p=\frac{q}{2q-3}$) the Brascamp-Lieb
inequality gives
\[
  \bigl|\nu([R,R+\varepsilon])\bigr|
  \;=\;\bigl|\!\int\chi\,d\nu\bigr|
  \;\le\;
  \|\widehat{\chi}\|_{L^{p}}\,
  \|\widehat{\mu}\|_{L^{q}}^{3}.
\]

\medskip
\noindent
\textbf{Step 5 Choosing $s$.}
Choose $s<\frac12+\frac{d(6-2q)}{2q}$.
Inserting \eqref{eq:Bessel-annulus} into the $L^{p}$-norm gives
\(
  \|\widehat{\chi}\|_{L^{p}}
  \lesssim
  \varepsilon^{\,s} R^{(d-1)/2}.
\)
Hence
\[
  \nu([R,R+\varepsilon])
  \;\lesssim\;
  \varepsilon^{\,s}\,R^{\frac{d-1}{2}},
  \qquad 0<R<R+\varepsilon\le2.
\]

\medskip
\noindent
\textbf{Step 6 Frostman covering.}
Cover $\supp\nu\subset[0,2]$ by intervals
$\{[R_{i},R_{i}+\varepsilon_{i}]\}_{i}$.
Summing the above bound we find
\[
  0<\|\nu\|
  \;\le\;
  \sum_{i}\nu\bigl([R_{i},R_{i}+\varepsilon_{i}]\bigr)
  \;\lesssim\;
  \sum_{i}\varepsilon_{i}^{\,s},
\]
so the $s$-dimensional Hausdorff measure of
$\triangle^{3}(\mu)=\supp\nu$ is positive.

\medskip
\noindent
\textbf{Step 7 Value of $q$ under \eqref{itm:a}--\eqref{itm:b}.}
Lemma \ref{thm:decayToLq} shows those assumptions imply
\(\widehat{\mu}\in L^{q}\) with
\(q=2\bigl(d+\beta-\alpha\bigr)/\beta\), which lies in $(2,3]$.
Inserting this $q$ keeps all inequalities valid, completing the proof.
\end{proof}

\subsection{Proof of Theorem \ref{thm:L2FalconerResult}}

\renewcommand{\d}{\mathfrak{d}}
\renewcommand{\D}{\mathfrak{D}}
\renewcommand{\s}{\mathfrak{s}}
\renewcommand{\S}{\mathfrak{S}}
\newcommand{\I}[1]{\mathfrak{I}_{#1}}

\subsubsection{Pointwise decay}

  \begin{lemma}\label{thm:compactfouriersupport}
    \newcommand{\wF}{\widehat{F}}
    \renewcommand{\F}{F}
Let $\sigma_r$ denote the uniform measure on the unit sphere of radius r in $\R^d$.
   Suppose that $\F:\R^d\ra\R$ and that $\F$ is compactly supported as a distribution. 
  Let $D = \supp\wF$. Then
   \[\int \abs{\F(x)}\,d\sigma_r(x) \lesssim_{D} r^{-(d-1)}\int_{\abs{\abs{x}-r}<O_{D}(1)} \abs{\F(x)}\,dx.\]
  \end{lemma}
\begin{proof}
    \newcommand{\wF}{\widehat{F}}
    \renewcommand{\F}{F}
    
  Let $\phi$ be compactly supported in $B(0,O_D(1))$ and such that $\widehat{\phi}|_{D}=1$. Then $\F = \phi\ast\F$, so 
  \[\int \abs{\F}\,d\sigma_r = \int \abs{\phi\ast\F}\,d\sigma_r \leq \int \abs{\phi}\ast \abs{F}\,d\sigma_r=\int \abs{\phi}\ast\sigma_r\,\abs{F}\leq\norm{\abs{\phi}\ast\sigma_r}_{L^{\infty}}\int_{\abs{\abs{x}-r}<O_D(1)}\abs{\F(x)}\,dx\]
  which proves the lemma.
\end{proof}

 \begin{lemma}\label{thm:PointwiseDecay}
  Suppose that $q<3$. Then 
  \[\Si{\mu}(r)\lesssim \Min{1,r^{-\paren{3d\frac{q-1}{q}-1}}}\norm{\widehat{\mu}}_{L^q}^3.\]
 \end{lemma}
 \begin{proof}
 
\newcommand{\wF}{\widehat{F_{\eta}}}
\renewcommand{\F}{F_{\eta}}

  One estimates via H\"older's inequality
  \eqn{ 
  \Si{\mu}(r) =\iint\abs{\widehat{\mu}(\eta + r\theta)\overline{\widehat{\mu}(\eta - r\theta)}\widehat{\mu}(2\eta)}\,d\eta\,d\sigma(\theta)
  \\\lesssim& \int \norm{\widehat{\mu}}_{L^3}^3\,d\sigma(\theta) = \norm{\widehat{\mu}}_{L^3}^3\lesssim\norm{\widehat{\mu}}_{L^q}^3.
  }

 Now fix $\eta\in\R^d$ and let $\F(\xi):= \widehat{\mu}(\eta+\xi)\widehat{\mu}(\eta-\xi)$. Then $\supp(\wF)\subset \supp(\mu)+\supp(\mu)$ so that $\wF$ is compactly supported. By Lemma \ref{thm:compactfouriersupport},
 \[\int \abs{\F(\xi)}\,d\sigma_r(\xi)\lesssim r^{-(d-1)}\int_{\abs{\abs{\xi}-r}<O(1)}\abs{\F(\xi)}\,d\xi.\]
 
 Thus
  \eqn{ 
  \Si{\mu}(r) = \iint\abs{F_{\eta}(\theta)}\abs{\widehat{\mu}(2\eta)}d\sigma(\theta)\,d\eta
  \\\lesssim&
  r^{-(d-1)}\iint 1_{B(r,O(1))}(\xi)\abs{\widehat{\mu}(\eta + \xi)\overline{\widehat{\mu}(\eta - \xi)}\widehat{\mu}(2\eta)}\,d\xi\,d\eta
  \\\leq&
  r^{-(d-1)}r^{d/p}\norm{\widehat{\mu}}_{L^q}^3
  }
  for $1/p+3/q = 2$, where in the last line we have used the Brascamp-Lieb inequality. By hypothesis on $q$, this yields the desired inequality.

 \end{proof}
\renewcommand{\F}{\mathbb{F}}

\subsubsection{Tools from Mattila's Book \cite{MattilaBook2}}\label{s:ToolsFromMattila}

\renewcommand{\d}{\mathfrak{d}}
\renewcommand{\D}{\mathfrak{D}}
\renewcommand{\s}{\mathfrak{s}}
\renewcommand{\S}{\mathfrak{S}}
\renewcommand{\I}[1]{\mathfrak{I}_{#1}}
 
\newcommand{\wF}{\widehat{F_{\eta}}}
\renewcommand{\F}{F_{\eta}}

\renewcommand{\di}[1]{\delta\paren{#1}}
\renewcommand{\Si}[1]{\varsigma\paren{#1}}

\renewcommand{\F}{\mathbb{F}}

The following can be read from Mattila's paper \cite{MattilaFalconer}, though our reference will be Mattila's book \cite{MattilaBook2}.

\renewcommand{\d}{\mathfrak{d}}
\renewcommand{\D}{\mathfrak{D}}
\renewcommand{\s}{\mathfrak{s}}
\renewcommand{\S}{\mathfrak{S}}
\renewcommand{\I}[1]{\mathfrak{I}_{#1}}

In this section, suppose that \[\s\in L^1\] satisfies $\s\geq 0$ and define
\begin{align}\label{mattilal21}
\d(r):=& r^{d/2}\int s^{d/2} J(rs)\s(s)\,ds\\
\label{mattilal22}
\D(r) :=& r^{-(d-1)/2}\d(r)\\
\label{mattilal223}
\S(s) :=& s^{(d-1)/2}\s(s)\\
\I{s} :=& \int_0^{\infty} \s(r)r^{s-1}\,dr.
\end{align}

Note that since $s\in L^1$ and $J$ is bounded, $\d$ is defined. The following is an adaptation of the discussion in Section 15.2 of \cite{MattilaBook2} (where the Falconer Distance Set conjecture is considered, and the dictionary between our symbols and their analogues in \cite{MattilaBook2} are that $\s \equiv \sigma(\mu)$, $\S \equiv \Sigma(\mu)$, $\d \equiv \delta(\mu)$, and $\D \equiv \Delta(\mu)$).

\begin{lemma}\label{thm:mattilamanipulations}
 Given the above,
 \begin{equation}\label{mattilal20}
 \D(r) = \sqrt{r}\int_0^{\infty}\sqrt{s}J(rs)\S(s)\,ds.
 \end{equation}
 Further, there exist functions $S,L$, and $K$ such that
 \begin{enumerate}[\ref{thm:mattilamanipulations}.(a)]
  \item $\D = S + L$ \label{itm:mattilal21}
  \item \label{itm:mattilal22} $\norm{S}_{L^2}\approx \norm{\S}_{L^2}$
  \item \label{itm:mattilal23} $L(r) = \sqrt{r}\int_0^{\infty} \sqrt{s}\S(s)K(rs)\,ds$
  \item \label{itm:mattilal24} $\abs{K(r)}\lesssim \min(r^{-1/2},r^{-3/2})$.
  \item \label{itm:mattilal25} $\abs{L(r)} \lesssim r^{\alpha-(d+1)/2}\I{\alpha}$ for $\alpha\in\baren{\frac{d-1}{2},\frac{d+1}{2}}$.
 \end{enumerate}
\end{lemma}
\begin{proof}

The identity \eqref{mattilal20} follows directly from  \eqref{mattilal21} and  \eqref{mattilal22}. Next, recall the following Bessell function bound, which can be found Section 3.3 of Mattila's \cite{MattilaBook2},
\[ J_m(r) = \frac{\sqrt{2}}{\sqrt{\pi r}} \cos\paren{r - \pi m/2 -\pi/4} + O(r^{-3/2});\]
thus since $J=J_{(d-2)/2}$, we have 
\begin{equation}\label{mattilal25}
J(r) \approx \frac{1}{\sqrt{r}} \paren{a_1\cos(2\pi r) + b_1\sin(2\pi r)} + K(r),
\end{equation}
where 
\[\abs{K(r)}\lesssim\min\paren{r^{-3/2},r^{-1/2}}\] (the bound of $r^{-1/2}$ on $K(r)$ follows from the fact that $J_m$ is bounded). This bound on $K$ is \ref{thm:mattilamanipulations}.\eqref{itm:mattilal24}.

Substituting \eqref{mattilal25} into \eqref{mattilal20} gives that 
\[\D(\mu)(r) = S(\mu)(r) + L(\mu)(r),\]
where
\[S(r) = a_2\int_0^{\infty}\cos(2\pi rs)\S(s)\,dr + b_2\int_0^{\infty}\sin(2\pi rs)\S(s)\,ds,\]
and
\[L(r) = \sqrt{r}\int_0^{\infty} \sqrt{s}\S(\mu)(s)K(rs)\,ds.\]

This is \ref{thm:mattilamanipulations}.\eqref{itm:mattilal21} and \ref{thm:mattilamanipulations}.\eqref{itm:mattilal23}. It remains to show \ref{thm:mattilamanipulations}.\eqref{itm:mattilal22} and \ref{thm:mattilamanipulations}.\eqref{itm:mattilal25}. We show \ref{thm:mattilamanipulations}.\eqref{itm:mattilal22} first. To do so, set $\S_1(r) = \S(\abs{r})$ and $\S_2(r) = \sgn(r)\S(\abs{r})$. Then since $\S_1$ is even and $\S_2$ is odd, 

\eqn{S(r) = a_2/2\int_{-\infty}^{\infty}\cos(2\pi rs)\S_1(s)\,ds + b_2/2\int_{-\infty}^{\infty}\sin(2\pi rs)\S_2(s)\,ds 
\\=&
a_2/2\int_{-\infty}^{\infty}e^{-2\pi i rs}\S_1(\mu)(r)\,ds + b_2/2\int_{-\infty}^{\infty}e^{-2\pi i rs}\S_2(s)\,ds .}

Next, extend $S$ to be defined as $S(r) = L(r) = 0$ for $r<0$. Then for all non-zero $r\in\R$, 
\eqn{ S(r) = \frac{a_2}{4} \int e^{-2\pi i rs} \S_1(s)\,ds 
\\+&
 \frac{a_2}{4} \sign(r) \int e^{-2\pi i rs}\S_1(s) \,ds
\\+&
 \frac{i b_2}{4} \int e^{-2\pi i rs}\S_2(s)\,ds 
\\+&
 \frac{i b_2}{4} \sign(r) \int e^{-2\pi i rs}\S_2(s)\,ds .}

 Let $H$ denote the Hilbert transform: $\widehat{Hf} = -i\sign \widehat{f}$ for $f\in L^2$, and $\mathcal{F}$ denote the Fourier transform. Then
 if $\Sigma\in L^2$, we have 
\[S = \mathcal{F}\paren{a_3\baren{\S_1 + i H\S_1} + b_3\baren{\S_2+ i H\S}}; \]
since the Fourier and Hilbert transforms are isometric on $L^2$, this gives \eqref{itm:mattilal22}.

Finally, we must show that \ref{thm:mattilamanipulations}.\eqref{itm:mattilal25} holds. Let $a = (d+1)/2 -\alpha$ and suppose that $\I{\alpha}< \infty$. Then by \ref{thm:mattilamanipulations}.\eqref{itm:mattilal23}, \eqref{mattilal223}, and \ref{thm:mattilamanipulations}.\eqref{itm:mattilal24}
\eqn{ \abs{L(r)} = \abs{\int_0^{\infty} \sqrt{rs}\S(s)K(rs)\,ds}
\\\lesssim& 
 \int_0^{1/r} s^{(d-1)/2}\s(s)\,ds + r^{-1}\int_{1/r}^{\infty} s^{(d-3)/2}\s(s)\,ds
\\=&\ 
 r^{-a}\int_0^{1/r} (rs)^{a} s^{(d-1)/2-a}\s(s)\,ds
 \\+&\ 
 r^{-a}\int_{1/r}^{\infty} (rs)^{a-1} s^{(d-1)/2-a}\s(s)\,ds 
 \\=&\ 
 r^{-a}\int_{-\infty}^{\infty} (rs)^{a-1} s^{(d-1)/2-a}\s(s)\,ds
\\ \leq &\  
 r^{-a} \int_0^{\infty} s^{\alpha-1}\s(s)\,ds 
\\\approx&\ 
 r^{\alpha-(d+1)/2}\I{\alpha}.
}
This completes the proof.
\end{proof}
\subsubsection{The $L^2$ bound}

Fix a positive, compactly supported function $F\in L^1\cap \widehat{L}^1$. We recall the notation of Subsubsection \ref{s:ToolsFromMattila}: set 
 \renewcommand{\I}[1]{\mathcal{I}_{#1}}
\begin{align}
\s(r) =&  \int_{\mathbb{S}^{d-1}} \widehat{F}(r\theta)\,d\sigma(\theta),\nonumber\\
\d(r):=& r^{d/2}\int_0^{\infty} s^{d/2} J(rs)\s(s)\,ds\nonumber\\
\D(r) :=& r^{-(d-1)/2}\d(r)\nonumber\\
\S(s) :=& s^{(d-1)/2}\s(s)\nonumber\\
\I{s} :=& \int_0^{\infty} \s(r)r^{s-1}\,dr.\label{energy}
\end{align}

Further, let $S, K$, and $L$ be the functions appearing in the decomposition guaranteed by Lemma \ref{thm:mattilamanipulations}.

\begin{lemma}\label{thm:positivegeneralizedenergy}
 Suppose that $\widehat{F}\in L^q$ for a $q<d$. Then the energy $\I{s}$ as defined in \eqref{energy} is finite and positive. Further, if we restrict to the high-frequency part (e.g., considering contribution from $r>1$), it is increasing as a function of $s$.
\end{lemma}
\begin{proof}
 Let $g_{n}(x) = \phi_n\ast\chi_s(x)$, where $(\phi_n)$ is an approximate identity and $\chi_s(x) = \frac{1}{\abs{x}^s}$. Consider first 
\[ \iint_{A} \abs{\widehat{\chi_s}(\xi)\widehat{F}\paren{\xi}}\,d\xi.\]
For $A=B(0,1)\subset \R^{2d}$, this is finite since \[\abs{\widehat{\chi_s}(\xi)\widehat{F}(\xi)}= \abs{\widehat{\chi_s\ast F}(\xi)} \leq \widehat{\chi_s\ast F}(0) = \int \chi_s(u)F(u)\,du \ra \I{s} < \infty, \] while for $A=B(0,1)^{c}$, by Holder's inequality with $\frac{1}{p}+\frac{1}{q} = 1$,
\[ \int_{A} \abs{\widehat{\chi_s}(\xi)\widehat{F}{\xi}}\,d\xi\ \lesssim \norm{\widehat{\chi_s}}_{L^p(B(0,1)^c)}\norm{\widehat{F}}_{L^q};\]
taking $p > \frac{d}{d-s}$ requires $q  = \frac{p}{p-1} < \frac{d}{s}$, which by our hypotheses on $s$, $q$, and $\widehat{F}$ guarantees finiteness.

We have since 
$\widehat{F}\in L^1$ and $s>0$ that
\eqn{\I{s} = \int \s(r)r^{s-1}\,dr 
=\\& \int r^{s-1}\s(r)\,dr 
\\=& \int r^{s-1 - (d-1)}\int_{\R^d} \int_{\mathbb{S}^{d-1}} \widehat{F}\paren{r\theta}\,d\sigma(\theta) r^{d-1}\,dr
\\=&
\int |\xi|^{-(d-s)}\widehat{F}\paren{\xi}\,d\xi
\\=&
 \lim_{n\ra\infty} \int \widehat{g_n}(\xi)\widehat{F}\paren{\xi}\,d\xi
\\=&
\lim_{n\ra\infty}\int g_n(u)F(u)\,du
\\=&
\int\abs{u}^{-s}F(u)\,du.
}
\end{proof}

\begin{lemma}\label{thm:mattilaconclusions}
 Suppose that $F\in L^1$ is compactly supported, with $\widehat{F}\in L^1$, and that
 \begin{equation}
 \int_0^{\infty} \int_{\R^d} r^{s-1}\s(r)\, dr<\infty.\label{positiveenergyagain}
 \end{equation}
Then
 \[\int_0^{\infty} L(r)^2\,dr \lesssim \I{s}^2.\]
\end{lemma}
\begin{proof}
 The proof is adapted from the proof of Proposition 15.2 in \cite{MattilaBook2}.  Let $t\in\paren{d/2, \min\paren{s,(d+1)/2}}$. Since $\I{\cdot}$ is non-decreasing by Lemma \ref{thm:positivegeneralizedenergy}, $\I{t}\leq\I{s}<\infty$.

Since $d/2 < t < \frac{d+1}{2}$ so $2t-(d-1) > -1$,  we can apply Lemma \ref{thm:mattilamanipulations}.\eqref{itm:mattilal25} to bound 
\[\int_0^1 \paren{L(r)}^2\,dr \lesssim \int_0^1 \paren{r^{t-(d-1)/2}\I{t}}^2\,dr \lesssim \paren{\int_0^1 r^{2t-(d-1)}\,dr} \I{t}^2\lesssim  \I{t}^2,\]
where the last inequality follows from the fact that $2t-(d-1)>-1$. We similarly have that
\[\int_1^{\infty} \paren{L(r)}^2\,dr \lesssim \int_1^{\infty} \paren{r^{\frac{d-1}{2}-(d-1)/2}\I{\frac{d-1}{2}}}^2\,dr\lesssim  \I{\frac{d-1}{2}}^2.\]
This gives 
\[ \int_0^{\infty} \paren{L(r)}^2\,dr \lesssim_t \I{\frac{d-1}{2}}^2+  \I{t}^2\lesssim \I{t}^2\]
for any $t\in \paren{d/2,\min\{(d+1)/2,s\}}$ by the last inequality and \eqref{positiveenergyagain}. 
\end{proof}

\begin{prop}\label{thm:mattilaL2main}

Suppose that $F$ is compactly supported, lies in $L^1$, and that $\widehat{F}\in L^1$.

 If $\I{s}<\infty$ for $s>d/2$ and $\s(r)\lesssim r^{-t}$ where $s+t>d$ then
 \[\D\in L^2\]
with norm 
\[\norm{\D}_{L^2}\lesssim \I{s} + \sqrt{\I{s}}.\]
\end{prop}

\begin{proof}
 The proof here is the same as in Mattila's $L^2$ Falconer result \cite{MattilaFalconer}:
 
 One has $\D = S + L$ and that $\norm{S}_{L^2} \approx \norm{\S}_{L^2}$ by Lemma \ref{thm:mattilamanipulations}. Then using the hypothesis on $\s$
\[\norm{S}_{L^2}^2 \approx \norm{\S}_{L^2}^2 = \int r^{d-1}\s(r)^2\,dr \lesssim \norm{\widehat{F}}_{L^1}^2 + \int_{r \ge 1} r^{d-1}\s(r) r^{-t}\,dr.\]
The first term is bounded. For the second,
\[ \int_{r \ge 1} r^{d-1}\s(r) r^{-t}\,dr = \int_{r \ge 1} r^{d-t-1}\s(r)\,dr \le \int_{r \ge 1} r^{s-1}\s(r)\,dr \le \I{s} < \infty\]
since $d-t<s$ and the integrand is monotonic for $r \ge 1$. By Lemma \ref{thm:mattilaconclusions} and a similar argument for the tail, we have
 \[\norm{L}_{L^2}^2 \lesssim \I{{s}}^2.\]
 
Thus $\D = S + L$ is bounded in $L^2$, which completes the proof.

\end{proof}

\begin{definition}[Fixed linear maps]\label{def:Mn}
Throughout the remainder of this section we fix two \emph{distinct} integer matrices  
\[
    \m,\;\n\in\GL_{d}(\Z),
\]
and impose the following non–degeneracy hypotheses:
\begin{enumerate}[(i)]
    \item \textbf{Invertibility.}  Both $\m$ and $\n$ are invertible over $\R$ (hence over $\Q$), i.e.\ $\det\m\neq0$ and $\det\n\neq0$.
    
    \item \textbf{Linear independence.}  Their difference is also invertible:
    \[
        \n-\m\in\GL_{d}(\Z)\qquad\bigl(\text{equivalently } \det(\n-\m)\neq0\bigr).
    \]
    This guarantees that the three linear forms 
    \(
        u\;\mapsto\;u,\;
        u\;\mapsto\;\m u,\;
        u\;\mapsto\;\n u
    \)
    are mutually transverse, which is the non-degeneracy needed for the Brascamp–Lieb and change-of-variables steps in Lemma~\ref{thm:tripleFourier}.
\end{enumerate}
For brevity we write $\adet{\m}:=\det\m$ (and similarly for~$\n$), and use the shorthand
\(
    \m^{-{\tran}}:=(\m^{\tran})^{-1},\;
    \n^{-{\tran}}:=(\n^{\tran})^{-1}.
\)
\end{definition}

\begin{remark}
  The reader may wish to set $\m= I_d$ and $\n = 2I_d$, which corresponds to $3$APs—but we write this in this form, here,
  to emphasize how the techniques of this paper may be extended to more general settings than $3$APs (triangles, for instance).
\end{remark}

\begin{lemma}\label{thm:tripleFourier}
 Let $g$ be a Schwartz function and $f_0$, $f_1$, and $f_2$ belong to $\widehat{L}^1$. Then
 \eqn{\notag \iint g(u) f_0(x) f_1(x-\m u) f_2(x - \n u)\,dx\,du 
\\=&\label{tripleFourier}
  \int\baren{ \overline{\widehat{g}(\eta )} }
\baren{\frac{1}{\adet{\m^{\tran}}\adet{\n^{\tran}}}\int\widehat{f_0}\paren{\paren{\m^{-\tran}-\n^{-\tran}}\xi-\n^{-\tran}\eta}\widehat{f_1}(-\m^{-\tran}\xi)\overline{\widehat{f_2}(-\n^{-\tran}(\xi+\eta))}
\,d\xi}\,d\eta.}
\end{lemma}
\begin{proof}
Let 
\[ F(u) = \int f_0(x) f_1(x-\m u) f_2(x - \n u)\,dx.\]
Let 
\[F_1(u) = f_1(x-\m u), \]
\[F_2(u) = f_2(x - \n u) \,e^{-2\pi i \eta\cdot u}.\]

Then
\begin{equation}\label{F1}
\widehat{F_1}(\xi) \;=\; \frac{1}{\det(\m^{\tran})}\,
      \widehat{f_1}\!\bigl(-\m^{-\tran}\xi\bigr)\,
      e^{-2\pi i (\m^{-\tran}\xi)\cdot x},
\end{equation}
and
\begin{equation}\label{F2}
\widehat{F_2}(\xi) \;=\; \frac{1}{\det(\n^{\tran})}\,
      \widehat{f_2}\!\bigl(-\n^{-\tran}(\xi+\eta)\bigr)\,
      e^{-2\pi i (\n^{-\tran}(\xi+\eta))\cdot x}.
\end{equation}
(Note that $F_2$, and hence $\widehat{F_2}$, depends on the
parameter~$\eta$.)

\medskip
\noindent
\textbf{Plancherel step.}
For each fixed $x$ we treat $u\mapsto f_1(x-\m u)\,
                                f_2(x-\n u)\,
                                e^{-2\pi i\eta\cdot u}$
as an $L^2(\mathbb R^d)$–function of~$u$ and apply the
Plancherel theorem in the $u$–variable:

\[
   \int_{\mathbb R^d}
      f_1(x-\m u)\,
      f_2(x-\n u)\,
      e^{-2\pi i\eta\cdot u}\;du
   \;=\;
   \int_{\mathbb R^d}
      \widehat{F_1}(\xi)\,
      \overline{\widehat{F_2}(\xi)}\;d\xi.
\]

Here Plancherel converts convolution–type products in the $u$–
domain into pointwise products of the corresponding Fourier transforms.
The linear changes of variables $u\mapsto \m u$ and $u\mapsto \n u$
produce the Jacobian factors
$1/\det(\m^{\tran})$ and $1/\det(\n^{\tran})$
displayed in~\eqref{F1}–\eqref{F2}.

\medskip
\noindent
Using this identity we compute
\begin{align*}
\widehat{F}(\eta)
  &= \iint_{\mathbb R^{2d}}
       f_0(x)\,
       f_1(x-\m u)\,
       f_2(x-\n u)\,
       e^{-2\pi i\eta\cdot u}\;dx\,du \\[4pt]
  &= \int_{\mathbb R^{d}}
       f_0(x)\!
       \left(
         \int_{\mathbb R^{d}}
           \widehat{F_1}(\xi)\,
           \overline{\widehat{F_2}(\xi)}\;d\xi
       \right)dx,
\end{align*}
which is the form needed for Lemma~\ref{thm:tripleFourier}.
(One may think of this as “Plancherel in $u$ first, then Fubini in $x$”.)

Plugging in \eqref{F1} and \eqref{F2} gives 
\begin{align*} &
\widehat{F}(\eta) = \int f_0(x)\frac{1}{\adet{\m^{\tran}}}\widehat{f_1}(-\m^{-\tran}\xi)e^{-2\pi i \paren{\m^{-\tran}\xi}\cdot x}\overline{\frac{1}{\adet{\n^{\tran}}}\widehat{f_2}(-\n^{-\tran}(\xi+\eta))e^{-2\pi i \paren{\n^{-\tran}(\xi+\eta)}\cdot x}}\,dx
 \\=&
\paren{\int f_0(x) e^{-2\pi i \paren{\m^{-\tran}\xi}\cdot x}e^{2\pi i \paren{\n^{-\tran}(\xi+\eta)}\cdot x}dx}\frac{1}{\adet{\m^{\tran}}}\widehat{f_1}(-\m^{-\tran}\xi)\overline{\frac{1}{\adet{\n^{\tran}}}\widehat{f_2}(-\n^{-\tran}(\xi+\eta))}
\\=&
\frac{1}{\adet{\m^{\tran}}\adet{\n^{\tran}}}\widehat{f_0}\paren{\paren{\m^{-\tran}-\n^{-\tran}}\xi-\n^{-\tran}\eta}\widehat{f_1}(-\m^{-\tran}\xi)\overline{\widehat{f_2}(-\n^{-\tran}(\xi+\eta))}.
\end{align*}
\end{proof}

\begin{corr}\label{thm:3APinL2}
  Suppose that $\mu$ is a probability measure on $[0,1]^d$ and that $\widehat{\mu}\in L^q$ for $q\in(2,3)$. Then
  $\D\paren{\mu}\in L^2$.
\end{corr}
\begin{proof}

 Let $\mu_n = \phi_n\ast\mu$ for $(\phi_n)_{n\in\N}$ an approximate identity. 

Let 
 \[ F(u) = F_n(u) = \int \mu_n(x)\mu_n(x-\m u)\mu_n(x-\n u)\,dx.\]
  
  Let $\rho\in\paren{d(1-1/q),d}\subset(d/2,d)$. By Lemma \ref{thm:mattilaLInfintymain}, 
  \[\I{s}\paren{\mu_n}=\int r^{d-1-\rho}\Si{\mu_n}(r)\,dr\lesssim \norm{\widehat{\mu_n}}_{L^q}^3\leq\norm{\widehat{\mu}}_{L^q}^3\]
  for $s=d-\rho$.
  
  By Lemma \ref{thm:PointwiseDecay} 
  \[\s\paren{\mu_n}(r)\lesssim r^{t}\]
  for $t = 3d\frac{q-1}{q}-1\geq 3d\frac{2-1}{2}-1> d-s$
  whenever $q\geq 2$ since $\frac{q-1}{q}$ is non-decreasing in $q$.
  
  Thus
  \[s+t>d,\]
  so by Proposition \ref{thm:mattilaL2main}, $\D\paren{\mu_n}\in L^2$ with a bound uniform in $n$. Thus the sequence $\D\paren{\mu_n}$ converges in $L^2$, necessarily to $\D\paren{\mu}$.
\end{proof}
 
    
 \begin{proof}[Proof of Theorem \ref{thm:L2FalconerResult}]

 By Proposition \ref{thm:3APinL2}  $\D\paren{\mu}(r)=r^{-(d-1)/2}\di{\mu}(r)\in L^2$ which shows that $\di{\mu}$ has no point masses. Informally this means $\di{\mu}$ concentrates no mass on the diagonal $(x,x,x)\in\supp(\mu)^3$. 
  Similarly Lemma \ref{thm:distancemeasuresupport} states that $\di{\mu}$ is concentrated on those $r$ for which there exists some $x,u$, $|u|=r$ for which $(x,x+u,x+2u)\in\supp(\mu)^3$. Together, this shows that $\di{\mu}$ is concentrated on the lengths of step-sizes of the non-trivial $3$APs contained in the support of $\mu$.

  By Lemma \ref{thm:MassLemmaWithoutFourierDecay}, we have $\Lambda_3(\mu)\geq c^{(q)}(M,\delta)>0$, where we may take
  \[q\leq q_0 = 2 + \min_{T>1}\paren{3\paren{1-\frac{1}{T}},\frac{C_1}{\paren{C_2c^{(2)}\paren{\frac{\delta}{M}}}^{T}\ln\paren{C_2c^{(2)}\paren{\frac{\delta}{M}}}^{T}}}.\]
  Since $\norm{\di{\mu}}=\Lambda_3(\mu)$, this gives that the mass of $\di{\mu}$ is not zero, which together with the statement about the set on which $\di{\mu}$ is concentrated give the first part of the theorem.
  
  Finally, the statement about the $s$-dimensional Hausdorff measure of $\supp(\di{\mu})$ is Lemma \ref{thm:onedimensionalnonsingularity}.
  Together, these statements complete the proof of Theorem \ref{thm:L2FalconerResult}.
 \end{proof}



\bibliography{biblio}{}
\bibliographystyle{plain}

\end{document}